\documentclass[11pt]{article}
\usepackage{amssymb,amsmath,amsthm,mathdesign}
\usepackage{pdfsync}
\usepackage{color}
\usepackage[colorlinks]{hyperref}

\newtheorem{theorem}{Theorem}[section]
\newtheorem{lemma}[theorem]{Lemma}
\newtheorem{proposition}[theorem]{Proposition}

\newtheorem{assumption}[theorem]{Assumption}
\newtheorem{remark}[theorem]{Remark}

\numberwithin{equation}{section}

\newcommand{\be}{\begin{equation}}
\newcommand{\ee}{\end{equation}}

\newcommand{\bes}{\begin{equation*}}
\newcommand{\ees}{\end{equation*}}

\def\E{\bE}

\def\cG{\mathcal{G}}

\newcommand{\sG}{\mathcal{G}}

\def\bE{\mathbb{E}}

\newcommand{\R}{\mathbb{R}}

\newcommand{\1}{\boldsymbol{1}}
\renewcommand{\d}{{\rm d}}

\renewcommand{\geq}{\geqslant}
\renewcommand{\leq}{\leqslant}
\renewcommand{\ge}{\geqslant}
\renewcommand{\le}{\leqslant}
\newcommand{\rd}{{\mathbb R^d}}

\def\m1{\mathbf{1}}

\allowdisplaybreaks[1]

\allowdisplaybreaks[1]

\title{Blow-up  results for space-time fractional stochastic partial differential equations}
\author{Sunday A. Asogwa\\ Auburn University
\and Jebessa B. Mijena\\Georgia College \& State University
\and Erkan Nane\\ Auburn University
 }
 \date{}
\begin{document}
\maketitle

\begin{abstract}
Consider non-linear time-fractional  stochastic reaction-diffusion  equations of the following type,
$$\partial^\beta_tu_t(x)=-\nu(-\Delta)^{\alpha/2} u_t(x)+I^{1-\beta}[b(u)+ \sigma(u)\stackrel{\cdot}{F}(t,x)]$$ in $(d+1)$ dimensions, where $\nu>0, \beta\in (0,1)$, $\alpha\in (0,2]$. The operator  $\partial^\beta_t$ is the Caputo fractional derivative while $-(-\Delta)^{\alpha/2} $ is the generator of an isotropic $\alpha$-stable L\'evy process and $I^{1-\beta}$ is the Riesz fractional integral operator. The forcing noise denoted by $\stackrel{\cdot}{F}(t,x)$ is a  Gaussian noise. These equations might be used as a model for materials with random thermal memory. We derive  non-existence (blow-up) of global random field solutions  under some additional conditions, most notably on  $b$, $\sigma$ and the initial condition. Our results complement those of P. Chow in \cite{chow2}, \cite{chow1}, and Foondun et al. in  \cite{Foondun-liu-nane}, \cite{foondun-parshad}  among others.
\end{abstract}
 {\bf Keywords:} space-time fractional stochastic partial differential equations, space-time white noise, space colored noise, finite time blow-up.

\section{Introduction and main results}
In the last two decades, many researchers have developed  interest in the study of  time fractional diffusion equations due to their immense applications in many applied and theoretical fields of science and engineering. A typical form of the time fractional diffusion equation is $\partial^\beta_tu=\nu \Delta u$  with $\beta\in (0,1)$ where $\Delta u$ denotes the Laplacian of $u$ and   $\partial^\beta_tu$ denotes the Caputo fractional derivative of $u$ defined by the expression
\begin{equation}\label{CaputoDef}
\partial^\beta_t u_t(x)=\frac{1}{\Gamma(1-\beta)}\int_0^t \frac{\partial
u_r(x)}{\partial r}\frac{dr}{(t-r)^\beta}.
\end{equation}

These time fractional equations are related with anomalous diffusions or diffusions in non-homogeneous media, with random fractal structures; see, for instance, \cite{mnx2013}. 

{
We can study the natural extension of the time-fractional diffusion equation
 \begin{equation}\label{tfspde-0}
 \partial^\beta_tu_t(x)=\Delta u_t(x)+\stackrel{\cdot}{W}(t,x);\ \  u_t(x)|_{t=0}=u_0(x),
 \end{equation}
 where $\stackrel{\cdot}{W}(t,x)$ is a   space-time white noise with $x\in \rd$.

The physically "correct" form of \eqref{tfspde-0} can be obtained using
{\bf time fractional Duhamel's principle} \cite{umarov-12}  as follows. Consider the time-fractional PDE with force term $f(t,x)$
 \begin{equation}\label{tfpde}
 \partial^\beta_tu_t(x)=\Delta  u_t(x)+f(t,x);\ \  u_t(x)|_{t=0}=u_0(x),
 \end{equation}
 whose solution is given by Duhamel's principle. The role of the external force $f(t,x)$ to the output can be seen as
  \begin{equation}
 \partial^\beta_tV(r, t,x)=\Delta V(r, t,x);\ \  V(r,r,x)= \partial^{1-\beta}_t f (t, x)|_{t=r},
 \end{equation}
  with solution
 $$
 V(t,r, x)=\int_{\rd}G_{t-r}(x-y) \partial^{1-\beta}_r f (r, y)dy,
 $$
where  $G_t(x)$ is the fundamental solution of $\partial_t^\beta u=\Delta  u$.
Thus \eqref{tfpde} has solution

$$u(t,x)=\int_\rd G_{t}(x-y)u_0(y)dy+\int_0^t\int_\rd G_{t-r}(x-y)\partial^{1-\beta}_r f(r,y)dydr.$$

 Now we will write the mild (integral) solution of \eqref{tfspde-0} using time fractional Duhamel's principle as the form (informally):
 \begin{equation}\label{mild-sol-der-noise}\begin{split}
u(t,x)&=\int_\rd G_t(x-y)u_0(y)dy\\
\ \ \ \ &+\int_0^t\int_\rd G_{t-r}(x-y) \partial^{1-\beta}_r[ \stackrel{\cdot}{W}(r,y)]dydr.
\end{split}\end{equation}

For $\gamma>0$, we define the Riesz fractional integral by
$$I^{\gamma} f(t):=\frac{1}{\Gamma(\gamma)} \int _0^t(t-\tau)^{\gamma-1}f(\tau)d\tau.$$
The Caputo fractional derivative  $\partial^\beta_t$ is the left inverse of  Riesz fractional integral $I^{\beta}$. That means, for every  $\beta\in (0,1)$, and   $h\in L^\infty(\R_+)$ or $h\in C(\R_+)$
$$ \partial _t^\beta I^\beta  h(t)=h(t).$$
The fractional derivative of the noise term  in \eqref{mild-sol-der-noise}  can now be removed as follows. Consider the time fractional PDE with a force given by $f(t,x)=I^{1-\beta}h(t,x)$, then by the time fractional  Duhamel's principle the  mild solution to
\eqref{tfpde} will be given by

\begin{equation*}
\begin{split}
u_t(x)&=\int_{\rd} G_t(x-y)u_0(y)dy+\int_0^t\int_{\rd} G_{t-r}(x-y)  \partial^{1-\beta}_r  I^{1-\beta}h(r,y)dydr.\\
&=\int_{\rd} G_t(x-y)u_0(y)dy+\int_0^t\int_{\rd} G_{t-r}(x-y) h(r,y)dydr.
\end{split}
\end{equation*}

Thus, the above discussion suggest  that the ``correct'' time fractional stochastic PDE is the
 following model problem:
 \begin{equation}\label{tfspde-1}
 \partial^\beta_tu_t(x)=\Delta u_t(x)+I^{1-\beta}[\stackrel{\cdot}{W}(t,x)];\ \  u_t(x)|_{t=0}=u_0(x).
 \end{equation}
{Using the Walsh  isometry the above fractional integral equation \eqref{tfspde-1}  is defined as
 $$\int_{\R^d} \phi(x)I^{1-\beta}[\stackrel{\cdot}{W}(t,x)]dx= \frac{1}{\Gamma(1-\beta)} \int_{\R^d}\int _0^t(t-\tau)^{-\beta}\phi(x) W(d\tau, dx),$$ for $\phi \in L^2(\R^d)$ . }

 By the Duhamel's principle, mentioned above,    mild (integral) solution of  \eqref{tfspde-1} will be  (informally):
\begin{equation}\label{mild-sol-tfspde}\begin{split}
u_t(x)&=\int_\rd G_t(x-y)u_0(y)dy+\int_0^t\int_{\R^d} G_{t-r}(x-y)W(dydr).
\end{split}
\end{equation}

{ Now we would like to give a {\bf Physical motivation} to study such time fractional stochastic PDEs which is adapted from \cite{chen-kim-kim-2014}.
The  type of fractional stochastic PDEs studied in the current paper can be used to model heat equation in a material with thermal memory.
If we let $u_t(x), k(t,x)$ and $\stackrel{\to}{H}(t,x)$ represent the body temperature, internal energy and flux density, respectively. Then using the following relations
\begin{equation}\begin{split}
\partial_tk(t,x)&=-div \stackrel{\to}{H}(t,x),\\
k(t,x)=\beta u_t(x), \ \ \stackrel{\to}{H}(t,x)&=-\lambda\nabla u_t(x),
\end{split}\end{equation}
we get the classical heat equation $\beta\partial_tu=\lambda\Delta u$.

According to the law of classical heat equation the speed of heat flow is infinite,
but since the heat flow can be disrupted by the response of the material the propagation speed can be finite.
In a material with thermal memory we might have
$$k(t,x)=\bar{\beta}u_t(x)+\int_0^tn(t-s)u_s(x)ds,$$
for some appropriate constant $\bar{\beta}$ and kernel $n$.  In most cases we would have $n(t)=\Gamma(1-\beta_1)^{-1}t^{-\beta_1}$ for $\beta_1\in (0,1)$.
The convolution gives the fact  that the nearer past affects the present more.
If in addition the internal energy also depends on past random effects, then
\begin{equation}\label{phys-1}
\begin{split}
k(t,x)&=\bar{\beta}u_t(x)+\int_0^tn(t-s)u_s(x)ds\\
\ \ \ &+\int_0^t l(t-s)h(s, u_s(x)) {W}(ds, x),
\end{split}\end{equation}
where $W$ is  ``white noise'' modeling the random effects.
Take $l(t)=\Gamma(2-\beta_2)^{-1}t^{1-\beta_2}$ for $\beta_2\in (0,1)$,
then after differentiation \eqref{phys-1} gives
\begin{equation}\label{tf-eqn}
\partial_t^{\beta_1}u=div\stackrel{\to}{F}+ \frac{1}{\Gamma(2-\beta_2)} \frac{\partial}{\partial t}\int _0^t(t-s)^{1-\beta_2}h(s,u_s(x)) W(ds,x).
\end{equation}
}
A  version of equation \eqref{tf-eqn} was studied recently by L. Chen and his co-authors: see, for example, \cite{le-chen-2019}

}
Mijena and Nane \cite{nane-mijena-2014} proposed to study a  class of space-time fractional stochastic heat type  equation as a physical model for heat in a material with random thermal memory.  In the current paper we consider the following related  space-time fractional stochastic reaction-diffusion type equations in $(d+1)$ dimension
\begin{equation}\label{intro:eq}
\partial^\beta_tu_t(x)=-\nu(-\Delta)^{\alpha/2} u_t(x)+I^{1-\beta}[b(u)+ \sigma(u)\stackrel{\cdot}{F}(t,x)]\ \ \
 t>0 \quad\text{and}\quad x\in \R^d,
\end{equation}
where $\nu>0, \beta\in (0,1)$, $\alpha\in (0,2]$. The operator  $\partial^\beta_t$ is the Caputo fractional derivative while $-(-\Delta)^{\alpha/2} $ denotes the fractional Laplacian, the generator of a $\alpha$-stable L\'evy process and $I^{1-\beta}$ is the Riesz fractional integral operator. The forcing noise denoted by $\stackrel{\cdot}{F}(t,x)$ is a  Gaussian noise and will be taken to be white in time and possibly colored in space. The initial condition will always be assumed to be a non-negative bounded deterministic function.  The functions  $\sigma$  and $b$ are  locally Lipschitz functions.

{
Taking the fractional integral of the noise term  in equation \eqref{intro:eq} is not  merely to get a simple integral solution. It is due the following important physical reason:
Taking the fractional derivative of order $1-\beta$ of both sides of the equation \eqref{intro:eq} gives the forcing function, in the traditional units $x/t$: see, for example, Meerschaert et al  \cite{Meerschaert-2016}. In this paper the authors work on a deterministic time fractional equation with an external force, but the same physical principle should apply for the stochastic equations too.

Recently a numerical approximation of solutions to space-time fractional stochastic equations was established in \cite{Julia-et-al-2017}. Versions of equation \eqref{intro:eq} with or without the fractional integral of the noise term was the subject of some papers recently: see, for example, \cite{allouba-spde1, allouba-spde2, chen-2016,chen-kim-kim-2014, chen-2016, le-chen-2019,cui-yan-11}.
}

{
Using the time fractional Duhamel's principle, as mentioned above, a mild solution to \eqref{intro:eq}  in the sense of Walsh \cite{walsh} is any $u$ which is adapted to the filtration generated by the Gaussian noise and satisfies the following evolution equation}
\begin{equation}\label{mild}
u_t(x)=
(\cG u)_t(x)+ \int_{\R^d}\int_0^t G_{t-s}(x-y)b(u_s(y))\d s\,\d y+\int_{\R^d}\int_0^t G_{t-s}(x-y)\sigma(u_s(y))F(\d s\,\d y),
\end{equation}
where
\begin{equation}\label{deter}
(\cG u)_t(x):=\int_{\R^d} G_t(x-y)u_0(y)\,\d y,
\end{equation}
and $G_t(\cdot)$ denotes the density of the time changed process $X_{E_t}$. More explanation about this  process is given in Section 2.

 The existence and uniqueness of the solution to \eqref{intro:eq} with the space-time white noise when  $d<(2\wedge \beta^{-1})\alpha$ has been studied by Mijena and Nane \cite{nane-mijena-2014} under global Lipschitz condition on $\sigma$, using the white noise approach of Walsh \cite{walsh}.  Foondun and Nane \cite{Foondun-nane} and  Foondun et al. \cite{Foondun-mijena-nane} established existence of solutions of space-time fractional equations with space colored noise. Asogwa and Nane \cite{Asogwa-Nane2016}, Mijena and Nane \cite{mijena-nane-2}, show that if $\sigma$ is globally Lipschitz, then for every non-negative measurable bounded initial function with non-empty compact support, solution to \eqref{intro:eq} is defined for all time and the distances to the origin of the farthest high peaks of absolute moments of solutions grow exactly linearly with time. See \cite{Asogwa-Nane2016,mijena-nane-2} for more details.  In this paper, we will be concerned with the moments of the solution.

If we further have
\begin{equation}\label{moments}
\sup_{x\in \R^d}\sup_{t\in[0,\,T]} \E|u_t(x)|^k<\infty \quad\text{for all}\quad L>0 \quad\text{and}\quad k\in[2,\,\infty],
\end{equation}
then we say that $u$ is a {\it random field} solution on $[0,T]$. Usually, existence is proved under the assumption that $\sigma$ is globally Lipschitz. But this can be proved under the local Lipschitz condition as well. We can see this by defining
\begin{equation*}
\tau_N:= \inf \{t > 0, \sup_{x\in \R^d}|u_t(x)| > N\},
\end{equation*}
then  we have $|\sigma(u_s(x))-\sigma(u_s(y))|\leq K_N |u_s(x)-u_s(y)|$ for any $s \leq \min(T ,\tau_N)$, where $K_N$ is a constant dependent on $N$. 
Following the techniques in \cite{Davar}, \cite{nane-mijena-2014} and \cite{walsh}, we can prove existence and uniqueness of a local solution in $(0, \min(T,\tau_N))$ provided that $0<\alpha<2$ and $d<(2\wedge \beta^{-1})\alpha$; two conditions which will be in force whenever we are dealing with the above equation.

When \eqref{intro:eq} has a solution $u_t(x)$ which is defined on $\mathbb{R}^{d}\times(0 , T)$ for {every} $T>0$, we say that the solution is global. The main aim of this paper is to show that under some additional conditions on the initial condition and the functions $\sigma$  and $b$, \eqref{intro:eq} cannot have global random field solutions. The failure of global solutions usually manifests itself via the `blow up' of certain quantities involving the solution.

The study of blow-up or non-existence of solutions has attracted a number of researches, because they are very useful to applied researchers. In this regard, Mueller and Sowers in \cite{Mueller2000,MuellerSowers1993} prove that the space-time white noise driven stochastic heat equation with  Dirichlet boundary condition will blow up in finite time with positive probability, if $\sigma(u) = u^{\gamma}$ with $\gamma >3/2$.
Bonder and Groisman in \cite{FernandezGroisman2009} also prove  the finite time blow-up for almost every initial data when nonnegative convex drift function satisfying $\int^{\infty} 1/f < \infty$ is taken into consideration. We refer the reader to  \cite{BaoYuan2016, chow2, chow1, Foondun-Parshad-15, Fujita, Foondun-liu-nane, LiPengJia2016, LvDuan2015} for more information on the blow-up phenomenon in the deterministic setting.

In this paper, we will work with white and  space colored noise driven equations. 
 First, we will look at the following equation driven by the space-time white noise
\begin{equation}\label{eq:white}
\partial^\beta_tu_t(x)=-\nu(-\Delta)^{\alpha/2} u_t(x)+I^{1-\beta}[ \sigma(u)\stackrel{\cdot}{W}(t,x)]\quad{t>0}\quad\text{and}\quad x\in \R^d,
\end{equation}
where $\stackrel{\cdot}{W}(t,x)$ is a space-time white noise.  We will also look at equations driven by  noise colored  in space of  the following type
\begin{equation}\label{eq:colored}
\partial^\beta_tu_t(x)=-\nu(-\Delta)^{\alpha/2} u_t(x)+I^{1-\beta}[ \sigma(u)\stackrel{\cdot}{F}(t,x)]\quad{t>0}\quad\text{and}\quad x\in \R^d,
\end{equation}
{where $\stackrel{\cdot}{F}(t,x)$  is a space colored noise.}
The corresponding mild solution in the sense of Walsh \cite{walsh} is given by
\begin{equation}\label{mild}
u_t(x)=
(\cG u)_t(x)+ \int_{\R^d}\int_0^t G_{t-s}(x-y)\sigma(u_s(y))F(\d s\,\d y).
\end{equation}
We will again be interested in the random field solution. But for this equation, we will need to impose some conditions on the noise term. We have
\begin{align*}
\E[\dot F(s,x)\dot F(t,y)]=\delta_0(t-s)f(x,\,y),
\end{align*}
where $f(x,\,y)\leq g(x-y)$ and $g$ is a locally integrable function on $\R^d$ with a possible singularity at $0$ satisfying
\begin{equation*}
\int_{\R^d}\frac{\hat{g}(\xi)}{1+|\xi|^\alpha}\d \xi<\infty,
\end{equation*}
where $\hat{g}$ denotes the Fourier transform of $g$.

It is worth mentioning that not a lot of work has been done in this type of problems for space-time fractional  stochastic partial differential equations.

 \begin{assumption}
The function $\sigma$ is a locally Lipschitz function satisfying the following growth condition. There exist a $\gamma>0$ such that
\begin{equation}\label{growth}
\sigma(x)\geq |x|^{1+\gamma}\quad\text{for all}\quad x\in \R^d.
\end{equation}
\end{assumption}

Now we are  ready to state our findings in detail. For the first couple  of our results, we will assume that the initial condition is bounded below by a positive constant given below
\begin{equation}\label{minumum-values-u0}
\inf_{x\in\R^d}u_0(x):=\kappa.
\end{equation}

\begin{theorem}\label{theo1}
Let $d<(2\wedge \beta^{-1})\alpha$. Suppose that $\kappa>0$ and $u_t$ be the solution to \eqref{eq:white}. Then there exists a $t_0>0$ such that for all $x\in \R^d$,
\begin{equation*}
\E|u_t(x)|^2=\infty\quad\text{whenever}\quad t\geq t_0.
\end{equation*}
\end{theorem}
This theorem  states that provided that the initial function is bounded below, the second moment will eventually be infinite for white noise driven equations.

{
\begin{remark}\label{remark-general-gamma}We can also get a blow up for the following equation  that was considered by Chen et al \cite{le-chen-2019} for any $\gamma>0$  and $d< 2\alpha+\frac\alpha\beta \min (2\gamma-1, 0)$.
\begin{equation}\label{eq:white-diff-frac-int}
\partial^\beta_tu_t(x)=-\nu(-\Delta)^{\alpha/2} u_t(x)+I^{\gamma}[ \sigma(u)\stackrel{\cdot}{W}(t,x)]\quad{t>0}\quad\text{and}\quad x\in \R^d,
\end{equation}
In this case the corresponding mild solution in the sense of Walsh \cite{walsh} is given by
\begin{equation}\label{mild}
u_t(x)=
(\cG u)_t(x)+ \int_{\R^d}\int_0^t H_{t-s}(x-y)\sigma(u_s(y))W(\d s\,\d y),
\end{equation}
where $H(t,x)$ is given by the time  fractional derivative of $G(t,x)$.  Note that when $\gamma=1-\beta$ this condition becomes the same condition in Theorem \ref{theo1}, $d<(2\wedge \beta^{-1})\alpha$. We can get finite time blow up as in the proof of Theorem \ref{theo1}. Some details of the proof of results in this remark are presented in section \ref{sect-3}.

\end{remark} }

 We have a slightly more complicated picture for equations with colored noise. We will assume the following non-degeneracy condition on the spatial correlation of the noise.
\begin{assumption}\label{color}
For fixed $R>0$,  there exists a positive number $K_{f}$ such that
\begin{equation*}
\inf_{x,\,y\in B(0,\,R)}f(x,\,y)\geq K_{f}.
\end{equation*}
\end{assumption}

Since we mostly set $R=1$ when using this condition, the dependence of $K_f$ on $R$ is not necessarily specified.
 The above   assumption is also very mild. There are a lot of examples including the Riesz Kernel, exponential kernel, Ornstein-Uhlenbeck-type kernels, Poisson kernels and Cauchy kernels; see, for example, Example 1.4 in \cite{Foondun-liu-nane} for more details.

\begin{theorem}\label{theo2}
Let $u_t$ be the solution to \eqref{eq:colored} and suppose that Assumption \ref{color} holds. Fix $t_0>0$, then there exists a positive number $\kappa_0$ such that for all $\kappa\geq \kappa_0$, and  $ x \in \R^d$ we have
\begin{equation*}
\E|u_t(x)|^2=\infty, \quad \text{whenever}\quad t\geq t_0.
\end{equation*}
\end{theorem}

To establish non-existence of the second moment, in contrast to  Theorem \ref{theo1},  we require that the initial condition is large enough. This is  a result of the spatially correlated nature of the noise, which induces some extra dissipation effect.
 In fact, even in the case of the corresponding linear equation($\sigma(u)\propto u$), it is known that for some correlation functions, their moments might not grow exponentially fast. See for instance \cite{ChenKim} and \cite{HuaLeNualart}. However, if we consider  the case when the correlation function is given by the Riesz Kernel, we have the following stronger result concerning the solution to \eqref{eq:colored}.

\begin{theorem}\label{Theo3}
Suppose that the correlation function $f$ is given by
\begin{align*}
f(x,\,y)=\frac{1}{|x-y|^\omega}\quad\text{with}\quad \omega<d\wedge(\alpha\beta^{-1}).
\end{align*}
Then for $\kappa>0$, there exists a positive number $\tilde{t}$ such that for all $t\geq \tilde{t}$ and $x\in \R^d$,
\begin{equation*}
\E|u_t(x)|^2=\infty.
\end{equation*}
\end{theorem}

{
\begin{remark}
We can also get a blow up for the following equation  that was considered by Chen et al \cite{le-chen-2019} for any $\gamma>0$ (where we also need to add the condition from Chen et al \cite{le-chen-2019}about d and other parameters!)
\begin{equation}\label{eq:white-diff-frac-int}
\partial^\beta_tu_t(x)=-\nu(-\Delta)^{\alpha/2} u_t(x)+I^{\gamma}[ \sigma(u)\stackrel{\cdot}{F}(t,x)]\quad{t>0}\quad\text{and}\quad x\in \R^d,
\end{equation}
In this case the corresponding mild solution in the sense of Walsh \cite{walsh} is given by
\begin{equation}\label{mild}
u_t(x)=
(\cG u)_t(x)+ \int_{\R^d}\int_0^t H_{t-s}(x-y)\sigma(u_s(y))F(\d s\,\d y).
\end{equation}
Where $H_t(x)$ is given by the time  fractional derivative of $G_t(x)$. Using Plancharel theorem and equation (4.8)  in \cite{le-chen-2019} we can get finite time blow up as in the proof of Theorem \ref{Theo3}.
We get  the following lower bound by using Plancharel theorem and equation (4.8)  in \cite{le-chen-2019}
\begin{equation*}
\int_{\R^d\times \R^d}H_{t-s}(x_1-y_1)H_{t-s}(x_2-y_2)f(y_1,\,y_2)\,\d y_1\d y_2\geq c(t-s)^{2\beta+2\gamma-2-\beta w/\alpha}.
\end{equation*}
Hence the nonlinear renewal inequality in this case is
$$
H(t)\geq C+C \int_0^t H(s)^{1+\eta} (t-s)^{2\beta+2\gamma-2-\beta w/\alpha},
$$
where $H(s)=\inf_{x\in \R^d, y\in \R^d} \E|u_s(x)u_s(y)|$. Similarly, we can get a blow-up result following the proof of  Remark \ref{remark-general-gamma}.

\end{remark} }

It is also important to mention that all the results established so far in this work are obtained under the assumption that the initial function is bounded below away from zero. In fact, as we shall see from the next result, this condition can be weakened.

\begin{assumption}\label{initial}
Suppose that initial condition is non-negative and satisfies the following,
\begin{equation*}
\int_{B(0,\,1)}u_0(x)\,\d x:=K_{u_0}>0.
\end{equation*}
\end{assumption}
We have taken $B(0,\,1)$ as a matter of convenience.

\begin{theorem}\label{energy-white}
Let $d<(2\wedge \beta^{-1})\alpha$, and $u_t$ be the solution to \eqref{eq:white}. Then, under Assumption \ref{initial}, there exists a $t_0\geq 0$ such that for all $t\geq t_0$ and $x\in \R^d$,
\begin{equation*}
\E|u_t(x)|^2=\infty\quad\text{whenever}\quad K_{u_0}\geq K,
\end{equation*}
where $K$ is some positive constant.
\end{theorem}
We have a similar result for the  equation driven by space colored noise.
\begin{theorem}\label{energy-colored}
Let $u_t$ be the solution to \eqref{eq:colored}. Then, under Assumptions \ref{color} and \ref{initial}, there exists a $t_0\geq 0$ such that for all $t\geq t_0$ and $x\in \R^d$,
\begin{equation*}
\E|u_t(x)|^2=\infty\quad\text{whenever}\quad K_{u_0}\geq K,
\end{equation*}
where $K$ is a positive constant.
\end{theorem}
It should be noted that the constant $K$ appearing in the above two results need not be the same.
The concept of our method involves obtaining non-linear renewal-type inequalities whose solutions blow up in finite time. We adapt the methods in \cite{Foondun-liu-nane} with crucial changes to suit to the space-time fractional equations. This method is soft and can be adapted to study a wider class of equations.
For the colored-noise case, a crucial quantity we study is   $\E|u_t(x)u_t(y)|$ instead of $\E|u_t(x)|^2$; and a good control of the deterministic term $(\mathcal{G}u)_t(x)$ is crucial in getting the non-existence of the solutions. Our methods depend on crucial  heat kernel estimates for short times and use the fact that we can restart the solution at a later time.
We will explain these methods in the  proof of our results.

Our next  theorem,  extends those of \cite{chow2} , \cite{chow1} and \cite{Foondun-liu-nane}. 
Fix $R>0$.  We will study the equations above  in the  ball $B(0,\,R)$  with Dirichlet boundary conditions. We will need  the following assumption.
\begin{assumption}\label{initial-dirichlet}
We assume that the initial condition $u_0$ is a non-negative function whose support, denoted by $S_{u_0}$ satisfies $B(0,\,R/2)\subset S_{u_0}$ such that $\inf_{x\in B(0,\,R/2)}u_0(x)>\tilde{\kappa}$ for some positive constant $\tilde{\kappa}$.
\end{assumption}

\begin{theorem}\label{Dirichlet}
Fix $R>0$ and consider
\begin{equation}\label{eq:dir}
 \partial^\beta_tu_t(x)=
 -\nu(-\Delta)^{\alpha/2} u_t(x)+I^{1-\beta}[ \sigma(u)\stackrel{\cdot}{F}(t,x)]\quad{t>0}\quad\text{and}\quad x\in B(0,\,R).
 \end{equation}
Here $-(-\Delta)^{\alpha/2}$ denotes the generator of a $\alpha$-stable L\'evy process killed upon exiting the ball $B(0,\,R)$. The noise $\dot F$, when not space-time white noise is taken to be spatially colored with correlation function satisfying all the conditions stated above. Fix $\epsilon>0$, then there exist $t_0>0$ and $K>0$, such that for $K_{u_0}>K$,
\begin{equation*}
\E|u_t(x)|^2=\infty\quad\text{for all}\quad t\geq t_0 \quad\text{and}\quad x\in B(0,\,R-\epsilon).
\end{equation*}
\end{theorem}

 Fujita in \cite{Fujita} showed that  the only global solution to \begin{equation*}
\partial_t u_t(x)=\Delta u_t(x)+u(x)^{1+\lambda}\quad\text{for}\quad t>0\quad x\in \R^d,
\end{equation*}
with initial condition $u_0$ and $\lambda>0$ is the trivial one for $\lambda <2/d.$ In the case $\lambda>2/d,$ the global solution exist for small enough $u_0$.  A good way to look at this result is that for large $\lambda$, the quantity $u^{1+\lambda}$ becomes much smaller when the initial condition is small and the dissipative effect of the Laplacian prevents the solution to grow too big for blow-up to happen. But, when $\lambda$ is close to zero, regardless of the size of the initial condition, the dissipative effect of the Laplacian cannot prevent blow up of the solution. For the reaction-diffusion type space-time fractional stochastic equations, we work with the first moment  $ \E(|u_t(x)|)$. There is still an interplay between the dissipative effect of the operator and the forcing term  and we are able to shed light only on part of the true picture.  We show that if the initial condition is large enough then there is no global solution. It might very well be just like for the deterministic case, if the non-linearity is high enough, then for small initial condition, there exist global solutions. See the survey papers  \cite{DL, Levine} for blow-up results for the deterministic equations.


Next we want to state our non-existence results for  reaction-diffusion type equations.

 \begin{assumption}\label{drift-polynomial}
The function $b$ is  locally Lipschitz satisfying the following growth condition. There exist a $\eta>0$ such that
\begin{equation}\label{growth}
b(x)\geq |x|^{1+\eta}\quad\text{for all}\quad x\in \R^d.
\end{equation}
\end{assumption}

\begin{theorem}\label{thm-additional-drift}
Suppose that $\sigma $ is globally Lipschitz and $b $ satisfies the conditions in Assumption \ref{drift-polynomial}.
Consider
\begin{equation}\label{eq:additional source}
 \partial^\beta_tu_t(x)=
 -\nu(-\Delta)^{\alpha/2} u_t(x)+I^{1-\beta}[ b(u_t(x))+\sigma(u)\stackrel{\cdot}{F}(t,x)]\quad{t>0}\quad\text{and}\quad x\in \mathbb{R}^d.
 \end{equation}
Here $-(-\Delta)^{\alpha/2}$ denotes the generator of  $\alpha$-stable L\'evy  process.  The noise $\dot F$, when not space-time white noise is taken to be spatially colored. 
{
 Then \eqref{eq:additional source} has no random field solution in the following  cases:

\noindent {\bf 1: } $\inf_{x\in \R^d}u_0(x)>\kappa>0$ and  $\eta>0$.


\noindent {\bf 2: } $||u_0||_{L^1(\R^d)}>0,$ and $\beta d\eta/\alpha<1$.


}
\end{theorem}
When $\beta=1$, a version of this theorem  with $\alpha=2$ was considered by  Chow  \cite{chow1} and a version with $\alpha\in (1,2)$ was considered by Foondun and Parshad \cite{foondun-parshad}.

The mild solution of equation \eqref{eq:additional source} is given in the sense of Walsh \cite{walsh} as follows:

\begin{equation}\label{mild-additional-drift}
u_t(x)=
(\cG u)_t(x)+ \int_{\R^d}\int_0^t b(u_s(y))G_{t-s}(x-y)\d s\,\d y+\int_{\R^d}\int_0^t G_{t-s}(x-y)\sigma(u_s(y))F(\d s\,\d y).
\end{equation}

{
\begin{remark}We can also get a blow up for the following equation  that was considered by Chen et al \cite{le-chen-2019} for any $\gamma>0$  and $d< 2\alpha+\frac\alpha\beta \min (2\gamma-1, 0)$.
\begin{equation}\label{eq:white-diff-frac-int}
\partial^\beta_tu_t(x)=-\nu(-\Delta)^{\alpha/2} u_t(x)+I^{\gamma}[ b(u_t(x)) +\sigma(u)\stackrel{\cdot}{F}(t,x)]\quad{t>0}\quad\text{and}\quad x\in \R^d.
\end{equation}
In this case the corresponding mild solution in the sense of Walsh \cite{walsh} is given by
\begin{equation}\label{mild}
u_t(x)=
(\cG u)_t(x)+ \int_{\R^d}\int_0^t b(u_s(y))H_{t-s}(x-y)\d s\,\d y+\int_{\R^d}\int_0^t H_{t-s}(x-y)\sigma(u_s(y))F(\d s\,\d y),
\end{equation}
where $H_t(x)$ is given by the time  fractional derivative of $G_t(x)$. Using Equation (4.7)  in \cite{le-chen-2019} for $\xi=0$ we get $\int_{\R^d} H_{t}(x)dx=t^{\beta+\gamma-1}$. We can show finite time blow up as in the proof of Theorem \ref{thm-additional-drift}.

Here the nonlinear renewal inequality becomes
$$
F(s)\geq C+C_1\int_0^tF(s)^{1+\eta}(t-s)^{\beta+\gamma-1}ds
$$
where  $F(s)=\inf_{x\in \R^d}\E |u_s(x)|^2$. A similar argument as in the proof of  Remark \ref{remark-general-gamma} can be used to get a blow-up result.

\end{remark} }



\begin{theorem}\label{thm-Dirichlet-additional-drift}
Suppose that $\sigma $ is globally Lipschitz and $b $ satisfies the conditions in Assumption \ref{drift-polynomial}.
Fix $R>0$ and consider
\begin{equation}\label{eq:dir-additional-source}
 \partial^\beta_tu_t(x)=
 -\nu(-\Delta)^{\alpha/2} u_t(x)+I^{1-\beta}[ b(u_t(x))+\sigma(u)\stackrel{\cdot}{F}(t,x)]\quad{t>0}\quad\text{and}\quad x\in B(0,\,R).
 \end{equation}
Here $-(-\Delta)^{\alpha/2}$ denotes the generator of  $\alpha$-stable  L\'evy process killed upon exiting the ball $B(0,\,R)$. The noise $\dot F$, when not space-time white noise is taken to be spatially colored. 
Let $\phi_1$ be the first eigenfunction of the fractional Laplacian with Dirichlet exterior boundary condition  in the ball $B=B(0,R)$.
{
Then \eqref{eq:dir-additional-source} has no random field solution in the following cases:



\noindent {\bf 1: } $\int_{B}u_0(x)\phi_1(x):=K_{u_0, \phi_1}>0,$ and $\beta (1+\eta)\leq 1$

\noindent {\bf 2: } $\beta (1+\eta)>1$ and for large $K_{u_0, \phi_1}>0$.

 }

\end{theorem}


The mild solution of equation \eqref{eq:dir-additional-source} is given in the sense of Walsh \cite{walsh} as follows:

\begin{equation}\label{eq:dir-mild-additional-drift}
\begin{split}
u_t(x)&=
(\cG_B u)_t(x)+ \int_{B(0,R)}\int_0^t b(u_s(y))G_B({t-s}, x,y)\d s\,\d y\\
&+\int_{B(0,R)}\int_0^t G_B({t-s},x,y)\sigma(u_s(y))F(\d s\,\d y),
\end{split}
\end{equation}
where $G_B({t}, x,y)$ is the density of $X_{E_t}$ killed on the exterior of $B$.

In this paper we will denote the ball of radius $R$ by $B=B(0,\,R)$. For $x\in \R^d$, $|x|$ will be the magnitude of $x$. The letter $c$ and $c^\ast$ with or without subscripts will denote a constant whose value is not relevant.

The outline of the article is the following. Preliminary notions  and results needed for the proofs of the main results are presented in Section \ref{sect-prelim}. Section \ref{sect-3} contains the proofs of Theorem \ref{theo1} and \ref{theo2}. Theorem \ref{Theo3} is proved in Section \ref{sect-4}.  In Section \ref{sect-5} we give the proofs of Theorems \ref{energy-white}, \ref{energy-colored}, and \ref{Dirichlet}. Finally, in section \ref{sect-drift} we present the proof of the remaining results. We list a couple of results we need  in the appendix, Section \ref{appendix}.



\section{Preliminaries}\label{sect-prelim}

Now we are ready to give results that are used in the proof of our  main results. Let $\alpha\in (0,2)$. Let $X_t$ denote a symmetric $\alpha$-stable L\'evy  process with density function denoted by $p(t,\,x)$. This is characterized through the Fourier transform which is given by

\begin{equation}\label{Eq:F_pX}
\widehat{p(t,\,\xi)}=e^{-t\nu|\xi|^\alpha}.
\end{equation}

Let $D=\{D_r,\,r\ge0\}$ denote a $\beta$-stable subordinator with $\beta\in (0,1)$ and $E_t$ be its first passage time. It is known that the density of the time changed process $X_{E_t}$ is given by  $G_t(x)$. By conditioning, we have

\begin{equation}\label{Eq:Green1}
G_t(x)=G(t,x)=\int_{0}^\infty p(s,\,x) f_{E_t}(s)\d s,
\end{equation}
where
\begin{equation}\label{Etdens0}
f_{E_t}(x)=t\beta^{-1}x^{-1-1/\beta}g_\beta(tx^{-1/\beta}),
\end{equation}
where $g_\beta(\cdot)$ is the density function of $D_1$ and is infinitely differentiable on the entire real line, with $g_\beta(u)=0$ for $u\le 0$. Moreover,
\begin{equation}\label{Eq:gbeta0}
g_\beta(u)\sim K(\beta/u)^{(1-\beta/2)/(1-\beta)}\exp\{-|1-\beta|(u/\beta)^{\beta/(\beta-1)}\}\quad\mbox{as}\,\, u\to0+,
\end{equation}
and
\begin{equation}\label{Eq:gbetainf}
g_\beta(u)\sim\frac{\beta}{\Gamma(1-\beta)}u^{-\beta-1} \quad\mbox{as}\,\, u\to\infty.
\end{equation}
We will also need the following estimates given in Lemma 2.1 in \cite{Foondun-nane},
 \begin{equation}\label{heat}
 c_1\bigg(t^{-\beta d/\alpha}\wedge \frac{t^\beta}{|x|^{d+\alpha}}\bigg)\leq G_t(x),
 \end{equation}
 and
 \begin{equation}\label{heat2}
 G_t(x)\leq c_2\bigg(t^{-\beta d/\alpha}\wedge \frac{t^\beta}{|x|^{d+\alpha}}\bigg),
 \end{equation}
 when $\alpha > d.$

Let $D\subset \rd$ be a bounded domain.
Let  $p_D(t,x,y)$ denote the heat kernel of the equation \eqref{eq:white} when $\beta=1$ and $\sigma=0$. This is the space fractional diffusion equation with Dirichlet exterior boundary conditions. A well known fact is that
\begin{equation}\label{upper-bound-killed-stable}
p_D(t,x,y)\leq p(s,x,y)\ \ \mathrm{for \ all}\  x, y\ \in D, t>0.
\end{equation}

Let $G_{D}(t,x,y)$ denote the heat kernel of the equation \eqref{eq:white} when $\sigma=0$. This is the space-time fractional diffusion equation with Dirichlet exterior boundary conditions.
 Using the representation from Meerschaert et al.  \cite{cmn-12} and \cite{mnv-09}
$$
G_D(t,x,y)=\int_0^\infty p_D(s,x,y)f_{E_t}(s)ds.
$$
and using equation \eqref{upper-bound-killed-stable}
we get
\begin{equation}\label{bounded-free-upper-bound}
G_D(t,x,y)\leq G(t,x,y) = G_t(x-y)\ \ \mathrm{for \ all}\  x, y\ \in D, t>0.
\end{equation}


We need the $L^2$-norm of the heat kernel given by the next Lemma.
\begin{lemma} [ Lemma 1 in \cite{nane-mijena-2014}]\label{Lem:Green1} Suppose that $d < 2\alpha$, then
\begin{equation}\label{Eq:Greenint}
\int_{{\R^d}}G^2_t(x)\d x  =C^\ast t^{-\beta d/\alpha},
\end{equation}
where the constant $C^{\ast}$ is given by
\begin{equation*}
C^\ast = \frac{(\nu )^{-d/\alpha}2\pi^{d/2}}{\alpha\Gamma(\frac d2)}\frac{1}{(2\pi)^d}\int_0^\infty z^{d/\alpha-1} (E_\beta(-z))^2 d z.
\end{equation*}
Here $E_\beta(x)$ is the Mittag-Leffler function defined by
\begin{equation}\label{ML-function}
E_\beta(x) = \sum_{k=0}^\infty\frac{x^k}{\Gamma(1+\beta k)}.
\end{equation}
\end{lemma}

The next proposition  is crucial in proving the lower bounds in Theorem \ref{Dirichlet}.

\begin{proposition}[Proposition 2.1 in \cite{Foondun-mijena-nane}]\label{density-lower-bound-killed-fractional}
Fix $\epsilon>0$, then there exists $t_0>0$ such that for all $x,y\in B(0, R-\epsilon)$ and for all $t<t_0$ and  $|x-y|<t^{\beta/\alpha}$ we have
$$
G_B(t,x,y)\geq C t^{-\beta d/\alpha},
$$
for some constant $C>0$.

\end{proposition}
\noindent For notational convenience, we set
\begin{equation*}
(\sG u)_t(x):=\int_{\R^d}G_t(x-y)u_0(y)\,\d y.
\end{equation*}
and
\begin{equation*}
(\tilde{\sG} u)_t(x):=\int_{\R^d}p(t,\,x-y)u_0(y)\,\d y.
\end{equation*}

\begin{proposition}\label{determ}
Let $x\in B(0,\,1)$ and Assumption \ref{initial} holds. Then there exists a positive number $t_0$ such that for $t\in (0,\,t_0]$, we have
\begin{equation*}
(\cG u)_{t+t_0}(x)\geq cK_{u_0},
\end{equation*}
where
\begin{equation}
K_{u_0}:=\int_{B(0,\,1)}u_0(x)\,\d x >0.
\end{equation}
\end{proposition}
\begin{proof}
By definition and Proposition 2.1 in \cite{Foondun-liu-nane}, we have
\begin{equation*}
\begin{aligned}
(\sG u)_{t+t_0}(x)&=\int_{\R^d}G_{t+t_0}(x-y)u_0(y)\,\d  y\\
&=\int_{\R^d}\int_0^\infty p(s,\,x-y)f_{E_{t+t_0}}(s)\,\d s\,u_0(y)\d  y\\
&=\int_0^\infty (\tilde{\sG} u)_s(x)f_{E_{t+t_0}}(s)\,\d  s\\
&\geq c_1K_{u_0}\int_{t_0}^{2t_0} f_{E_{t+t_0}}(s)\,\d  s\\
& = c_1K_{u_0}\int_{t_0}^{2t_0}(t+t_0)\beta^{-1}s^{-1-1/\beta}g_{\beta}((t+t_0)s^{-1/\beta})\,\d s\\
&= c_1K_{u_0}\displaystyle\int_{\frac{t+t_0}{(2t_0)^{1/\beta}}}^{\frac{t+t_0}{t_0^{1/\beta}}}g_{\beta}(u)\,\d u\\
&\geq c_1K_{u_0}\displaystyle\int_{\frac{2t_0}{(2t_0)^{1/\beta}}}^{\frac{t_0}{t_0^{1/\beta}}}g_{\beta}(u)\,\d u\\
&= c_2K_{u_0},
\end{aligned}
\end{equation*}
where $c_2$ depends on $t_0$. The last equality before the last inequality follows by substitution.
\end{proof}
\noindent Define
\begin{equation*}
(\cG_Bu)_t(x):=\int_{B(0,\,R)}G_{B}(t, x,y)u_0(y)\,\d y.
\end{equation*}
The following proposition will be used in the proof of Theorem \ref{Dirichlet}.
\begin{proposition}\label{determ-dirich}
Let $t\leq \left(\frac{R}{2}\right)^\alpha$ and $R>0$. Under Assumption \ref{initial-dirichlet}, we have
\begin{align*}
(\cG_Bu)_t(x)\geq c,\quad\text{for all}\quad x\in B(0,\,R/2),
\end{align*}
where $c$ is some positive constant.
\end{proposition}
\begin{proof}
By a simple conditioning and Proposition 2.2 in \cite{Foondun-liu-nane}, we have
\begin{align*}
(\cG_Bu)_t(x)&=\int_{B(0,\,R)}G_{B}(t,x,y)u_0(y)\,\d y\\
&=\int_0^\infty (\tilde{\mathcal G}_B u)_{s}(x)f_{E_t}(s)\,\d s\\
&\geq \int_0^{(R/2)^\alpha} (\tilde{\mathcal G}_B u)_{s}(x)f_{E_t}(s)\,\d s\\
&\geq c_1\int_0^{(R/2)^\alpha}f_{E_t}(s)\,\d s\\
&=c_1\int_0^{(R/2)^\alpha}t\beta^{-1}s^{-1-1/\beta}g_{\beta}(ts^{-1/\beta})\,\d s\\
&= c_1\int^{\infty}_{t/(R/2)^{\alpha/\beta}}g_{\beta}(u)\,\d u\\
&\geq c_1\int^{\infty}_{(R/2)/(R/2)^{\alpha/\beta}}g_{\beta}(u)\,\d u\\
&= c.
\end{align*}
\end{proof}

\begin{proposition}\label{prop-kernel}
Suppose that $t\leq \left(\frac{R}{2}\right)^\alpha$
 and $R>0$. Then for all $x_1,\,x_2 \in B(0,\,R)$, we have
\begin{equation*}
\int_{B(0,\,R)\times B(0,\,R)}G_{t-s}(x_1-y_1)G_{t-s}(x_2-y_2)f(y_1,\,y_2)\,\d y_1\d y_2\geq cK_{f},
\end{equation*}
where $s\leq t$ and $c$ is some positive constant.
\end{proposition}
\begin{proof}
By assumption \ref{color}, we observe
\begin{align*}
\int_{B(0,\,R)\times B(0,\,R)}&G_{t-s}(x_1-y_1)G_{t-s}(x_2-y_2)f(y_1,\,y_2)\,\d y_1\d y_2\\
&\geq K_f\int_{B(0,\,R)\times B(0,\,R)}G_{t-s}(x_1-y_1)G_{t-s}(x_2-y_2)\,\d y_1\d y_2.
\end{align*}
Let
\begin{equation*}
\mathcal{A}_i:=\{y_i\in B(0,\,R); |x_i-y_i|\leq (t-s)^{\beta/\alpha} \},\quad \text{for}\quad i = 1,2.
\end{equation*}
Since $t\leq \left(\frac{R}{2}\right)^\alpha$ we observe that $|\mathcal{A}_i|=c|t-s|^{\beta d/\alpha}$ for some constant $c$. Using the estimates given by \eqref{heat} for $G_t(x)$, we have
\begin{align*}
K_f\int_{B(0,\,R)\times B(0,\,R)}&G_{t-s}(x_1-y_1)G_{t-s}(x_2-y_2)\,\d y_1\d y_2\\
&\geq K_f\int_{\mathcal{A}_1\times \mathcal{A}_2}G_{t-s}(x_1-y_1)G_{t-s}(x_2-y_2)\,\d y_1\d y_2\\
&\geq K_f\int_{\mathcal{A}_1\times \mathcal{A}_2}C(t-s)^{-2\beta d/\alpha}\,\d y_1\d y_2\\
&= c_2K_f.
\end{align*}
This completes the proof.
\end{proof}
 We need to  introduce a version of the above result for the killed space-time fractional kernel using Proposition \ref{density-lower-bound-killed-fractional}.
\begin{proposition}\label{prop-kernel-killed}
Let $R>0$ and fix $\epsilon>0$. Then for all $x_1,\,x_2 \in B(0,\,R-\epsilon)$ and $t\leq \left(\frac{R}{2}\right)^\alpha$, we have
\begin{equation*}
\int_{B(0,\,R)\times B(0,\,R)}G_{B}( t-s, x_1-y_1)G_{B} (t-s, x_2-y_2)f(y_1,\,y_2)\,\d y_1\d y_2\geq cK_{f},
\end{equation*}
where $s\leq t$ and $c$ is some positive constant.
\end{proposition}
\begin{proof}
Assumption \ref{color} gives
\begin{align*}
\int_{B(0,\,R)\times B(0,\,R)}&G_{B}(t-s, x_1-y_1)G_{B}( t-s, x_2-y_2)f(y_1,\,y_2)\,\d y_1\d y_2\\
&\geq K_f\int_{B(0,\,R-\epsilon)\times B(0,\,R-\epsilon)}G_{B}( t-s, x_1-y_1)G_{B}( t-s, x_2-y_2)\,\d y_1\d y_2.
\end{align*}
Since $t\leq \left(\frac{R}{2}\right)^\alpha$ if we set
\begin{equation*}
\mathcal{A}_i:=\{y_i\in B(0,\,R-\epsilon); |x_i-y_i|\leq (t-s)^{\beta/\alpha} \}\quad \text{for}\quad i = 1, 2,
\end{equation*}
then $|\mathcal{A}_i|=c_1|t-s|^{\beta d/\alpha}$ for some $c_1$. We therefore have using Proposition \ref{density-lower-bound-killed-fractional}
\begin{align*}
\int_{B(0,\,R-\epsilon)\times B(0,\,R-\epsilon)}&G_{B} ( t-s, x_1-y_1)G_{B}( t-s, x_2-y_2)\,\d y_1\d y_2\\
&\geq \int_{\mathcal{A}_1\times \mathcal{A}_2}G_{B}( t-s, x_1-y_1)G_{B}( t-s, x_2-y_2)\,\d y_1\d y_2\\
&\geq \int_{\mathcal{A}_1\times \mathcal{A}_2}C(t-s)^{-2\beta d/\alpha}\,\d y_1\d y_2\\
&= c_2,
\end{align*}
This proves the required inequality.
\end{proof}

\begin{remark}
Under the same assumption of Proposition \ref{prop-kernel-killed}, we clearly have
\begin{equation*}
\int_{B(0,\,R-\epsilon)\times B(0,\,R-\epsilon)}G_{B}( t-s, x_1-y_1)G_{B}( t-s,x_2-y_2)f(y_1,\,y_2)\,\d y_1\d y_2\geq cK_{f},
\end{equation*}
where $s\leq t$ and $c$ is some positive constant.

\end{remark}

In the next two propositions we will give the renewal inequalities needed to prove non-existence results.
\begin{proposition}\label{volterra}
Fix $T>0$ and suppose that $h$ is a non-negative function satisfying the following non-linear integral inequality,
\begin{align*}
h(t)\geq C+D\int_0^t\frac{h(s)^{1+\gamma}}{(t-s)^{d\beta/\alpha}}\,\d s,\quad \text{for}\quad 0<t\leq T,
\end{align*}
where $C$, $D$ and $\gamma$ are positive numbers. Then for any $t_0\in (0,\,T]$, there exists an $C_0$ such that for $C>C_0$
\begin{align*}
h(t)=\infty\quad\text{whenever}\quad{t\geq t_0}.
\end{align*}
\end{proposition}
\begin{proof}
Since $t\leq T$ the inequality reduces to
\begin{align*}
h(t)\geq C+\frac{D}{T^{d\beta/\alpha}}\int_0^th(s)^{1+\gamma}\,\d s.
\end{align*}
Thanks to comparison principle, it suffices to consider
\begin{align*}
h(t)=C+\frac{D}{T^{d\beta/\alpha}}\int_0^th(s)^{1+\gamma}\,\d s,\quad \text{for}\quad t\leq T,
\end{align*}
which is equivalent to the following non-linear ordinary differential equation,
\begin{equation*}
\frac{h'(t)}{h(t)^{1+\gamma}}=\frac{D}{T^{d\beta/\alpha}},
\end{equation*}
with initial condition $h(0) = C$, whose solution is given by
\begin{equation*}
\frac{1}{h(t)^{\gamma}}=\frac{1}{C^\gamma}-\frac{\gamma Dt}{T^{d\beta/\alpha}}, \quad\text{for}\quad t\leq T.
\end{equation*}
Thus the blowup occurs at $t=\frac{T^{d\beta/\alpha}}{C^\gamma D\gamma}$. Choose $C>\left(\frac{T^{d\beta/\alpha}}{D\gamma t_0}\right)^{1/\gamma}$ for any fixed $t_0<T$. The conclusion follows since $h(t)$ is increasing on $(0,\infty)$ and blow-up occurs before time $t_0$. \end{proof}
Next we will give a slightly modified renewal inequalities needed for the proof of our main results.

\begin{proposition}\label{rem-volterra}
Let $0<(1+\gamma)d\beta/\alpha<1$. Suppose $h$ is a non-negative function satisfying the following non-linear integral inequality,
\begin{align*}
h(t)\geq C+D\int_0^t\frac{h(s)^{1+\gamma}}{(t-s)^{d\beta/\alpha}}\,\d s,\quad \text{for}\quad t>0,
\end{align*}
where $C$, $D$ and $\gamma$ are positive numbers. Then for any $C>0$ there exists $t_0>0$ such that $h(t)=\infty$ for all $t\geq t_0$.
\end{proposition}
\begin{proof}

Since $0<t-s \leq t$, we get
\begin{align*}
h(t)\geq C+D\int_0^t\frac{h(s)^{1+\gamma}}{t^{d\beta/\alpha}}\,\d s,\quad \text{for}\quad t>0.
\end{align*}
Now let $q(t):=t^{d\beta/\alpha}h(t)$ and since we can always assume $t_0>1$, the above inequality becomes
\begin{align*}
q(t)\geq C+D\int_0^t\frac{q(s)^{1+\gamma}}{s^{(1+\gamma)d\beta/\alpha}}\,\d s,\quad \text{for}\quad t\geq1.
\end{align*}
We only need to consider the following ordinary differential equation,
\begin{equation*}
\frac{q'(t)}{[q(t)]^{1+\gamma}}=\frac{D}{t^{(1+\gamma)d\beta/\alpha}}, \quad\text{for}\quad t\geq 1,
\end{equation*}
with initial condition $q(1)=C$. The solution of this equation is given by
$$\frac{1}{q(t)^{\gamma}} = \frac{1}{C^\gamma} +\frac{\gamma D}{1-\frac{(1+\gamma)d\beta}{\alpha}}- \frac{\gamma Dt^{1-\frac{(1+\gamma)d\beta}{\alpha}}}{1-\frac{(1+\gamma)d\beta}{\alpha}},\quad \text{for}\ t \geq 1.$$
Since $(1+\gamma)d\beta/\alpha<1$ the blowup occurs when $t$ is equal to $t_0$ given by
$$t_0: = \left(\frac{1-\frac{(1+\gamma)d\beta}{\alpha}}{\gamma D}\left(\frac{1}{C^\gamma} + \frac{\gamma D}{1-\frac{(1+\gamma)d\beta}{\alpha}}\right)\right)^{1/(1-\frac{(1+\gamma)d\beta}{\alpha})}.$$
Thus, $h(t) = \infty$ for $t\geq t_0$ since $h(t)$ is increasing on $(0, \infty).$

\end{proof}

\begin{remark}
The above Proposition \ref{rem-volterra} is also true when $h$ satisfies
\begin{align*}
h(t)\geq Ct^{-d\beta/\alpha} + D\int_{0}^{t}\frac{h(s)^{1+\gamma}}{(t-s)^{d\beta/\alpha}}\,\d s,\ \text{for}\ t>0.\end{align*}
In this case $t_0$ is given by

$$t_0 := \left(\frac{1-\frac{(1+\gamma)d\beta}{\alpha}}{\gamma DC^\gamma}\right)^{1/(1-\frac{(1+\gamma)d\beta}{\alpha})}.$$

\end{remark}
\begin{remark}\label{rem:gamma-greater-than1}
The above proposition holds when $(1+\gamma)/\alpha\geq1$ as well. This is because we can always write $\gamma=\gamma_0+(\gamma-\gamma_0)$ so that $\gamma_0<\gamma$ and $(1+\gamma_0)/\alpha<1$. Now we use the fact that $h(t)>A$ for all $t>0$ to reduce the integral inequality to
\begin{align*}
h(t)\geq C+ D\int_{0}^{t}\frac{h(s)^{1+\gamma}}{(t-s)^{d\beta/\alpha}}\,\d s\geq C+DC^{\gamma-\gamma_0}\int_0^t\frac{h(s)^{1+\gamma_0}}{(t-s)^{d\beta/\alpha}}\,\d s,\quad \text{for}\quad t>0.
\end{align*}
The result now follows by Proposition \ref{rem-volterra}.
\end{remark}

Next we state a general version of the proposition \ref{rem-volterra}.
\begin{proposition}\label{rem-volterra-general}
Let $0<\theta$. Suppose $h$ is a non-negative function satisfying the following non-linear integral inequality,
\begin{align*}
h(t)\geq C+D\int_0^t\frac{h(s)^{1+\gamma}}{(t-s)^{\theta}}\,\d s,\quad \text{for}\quad t>0
\end{align*}
where $C$, $D$ and $\gamma$ are positive numbers. Then for any $C>0$ there exists $t_0>0$ such that $h(t)=\infty$ for all $t\geq t_0$.
\end{proposition}
\begin{proof}
The proof is similar to the proof of Proposition \ref{rem-volterra} and Remark \ref{rem:gamma-greater-than1}. So it is omitted here.
\end{proof}

\section{Proof of Theorems \ref{theo1} and \ref{theo2}} \label{sect-3}
\begin{proof}[\bf Proof of Theorem \ref{theo1}.]
To start  the proof of the theorem with the use of the  mild formulation of the solution given by \eqref{mild}, then take second moment and use the Walsh isometry to get
\begin{align*}
\E|u_t(x)|^2&=|(\cG u)_t(x)|^2+\int_0^t\int_{\R^d} G^2_{t-s}(x-y)\E|\sigma(u_s(y))|^2\,\d y\,\d s\\
&:=I_1+I_2.
\end{align*}
Using the fact that the initial condition is bounded below gives
\begin{align*}
I_1\geq \kappa^2.
\end{align*}
This follows since $\int_{\rd}G_t(x-y)dy=1$.
By utilizing the growth condition on $\sigma,$ Jensen's inequality, and Lemma \ref{Lem:Green1} we bound $I_2$ as follows
\begin{align*}
I_2&\geq \int_0^t\left(\inf_{x\in {\R^d}}\E|u_s(x)|^2\right)^{\1+\gamma}\int_{\R^d} G^2_{t-s}(x-y)\,\d y\,\d s\\
&\geq C^\ast\int_0^t\left(\inf_{x\in {\R^d}}\E|u_s(x)|^2\right)^{\1+\gamma}\frac{1}{(t-s)^{\beta d/\alpha}}\d s.
\end{align*}
If we set
\begin{align*}
P(s):=\inf_{x\in \R^d}\E|u_s(x)|^2,
\end{align*}
the inequalities becomes
\begin{equation*}
P(t)\geq C+C^{\ast}\int_0^t\frac{P(s)^{1+\gamma}}{(t-s)^{d\beta/\alpha}}\,\d s.
\end{equation*}
Proposition \ref{volterra} now completes the proof of the theorem.
\end{proof}

\begin{proof}[Proof of Remark \ref{remark-general-gamma}]
Using Lemma 5.5 in \cite{le-chen-2019} we can get finite time blow up as in the proof of Theorem \ref{theo1}.

Here the nonlinear renewal inequality becomes
\begin{equation}\label{general-gamma}
P(s)\geq C+C_1\int_0^tP(s)^{1+\eta}(t-s)^\theta ds
\end{equation}
where  $P(s)=\inf_{x\in \R^d}\E |u_s(x)|^2$ and $\theta = 2(\beta+\gamma-1)-\d\beta/\alpha$. We have blow-up for $\theta<0$ by Proposition \ref{rem-volterra-general}. If $\theta>0$, we can show blow-up as follows: First note that $P(t)>C>0$. Suppose $P(t)<\infty$ for all $t$. Taking Laplace transform of both sides of \eqref{general-gamma} we get
\begin{equation*}
\begin{aligned}
\mathcal{L}\{P(t)\}(\lambda)) \geq \frac{C}{\lambda} + C_1\mathcal{L}\{P^{1+\eta}(t)\}(\lambda)\mathcal{L}\{t^\theta\}(\lambda),
\end{aligned}
\end{equation*}
where $\mathcal{L}\{P(t)\}(\lambda)) = \int_{0}^\infty e^{-\lambda t}P(t)dt$ represent Laplace transform. Using Jensen's inequality we have
\begin{eqnarray*}
 \mathcal{L}\{P^{1+\eta}(t)\}(\lambda) &=&\lambda \int_{0}^\infty P^{1+\eta}(t)\left(\frac{e^{-\lambda t}}{\lambda}\right)dt\\
 &\geq&\lambda\left(\int_{0}^\infty P(t)\left(\frac{e^{-\lambda t}}{\lambda}\right)dt\right)^{1+\eta}\\
 &=&\frac{1}{\lambda^\eta}\left[\mathcal{L}\{P(t)\}(\lambda)\right]^{1+\eta}.
\end{eqnarray*}
So we get
$$\mathcal{L}\{P(t)\}(\lambda)) \geq \frac{C}{\lambda} +\frac{C_1\Gamma(\beta+1)\left[\mathcal{L}\{P(t)\}(\lambda)\right]^{1+\eta}}{\lambda^{\eta+\theta+1}}\geq\frac{C^*\left[\mathcal{L}\{P(t)\}(\lambda)\right]^{1+\eta}}{\lambda^{\eta+\theta+1}}.$$
Thus,
$$\mathcal{L}\{P(t)\}(\lambda)) \leq \lambda^{(\eta+\theta+1)/\eta}\rightarrow 0\ \  \text{as}\  \lambda\rightarrow 0.$$
That means,
$$\lim_{\lambda\rightarrow 0}\mathcal{L}\{P(t)\}(\lambda)) =\lim_{\lambda\rightarrow 0} \int_{0}^{\infty}e^{-\lambda t}P(t)dt = \int_{0}^{\infty}P(t)dt = 0.$$
Hence, $P(t) = 0$ a.s. which contradicts $P(t)>C>0.$
\end{proof}

For  the proof of Theorem \ref{theo2} we need the following proposition.
\begin{proposition}\label{prop-colored}
Suppose that there exists a $\kappa_0>0$ and $t_0>0$ such that the lower bound of $u_0$ in \eqref{minumum-values-u0} satisfies $\kappa>\kappa_0$.  Then  for all $t_0<t\leq (1/2)^\alpha$, we have
\begin{align*}
\E|u_t(x)u_t(y)|=\infty\quad\text{for all},\quad   x,\,y\in B(0,\,1).
\end{align*}
\end{proposition}
\begin{proof}
By the mild formulation \eqref{mild} we observe
\begin{align*}
\E|&u_t(x)u_t(y)|\\
&\geq \cG u_t(x)\cG u_t(y)+\int_0^t\int_{\R^d\times\R^d}G_{t-s}(x-z)G_{t-s}(y-w)f(z,w)\left(\E|u_s(z)u_s(w)|\right)^{1+\gamma}\,\d z\d w\d s\\
&:=I_1+I_2.
\end{align*}
First consider the term $I_1$. Using  the fact that the initial condition is bounded below by $\kappa$ gives
\begin{equation*}
I_1\geq \kappa^2.
\end{equation*}
By Proposition \ref{prop-kernel} for $t<\left(\frac{1}{2}\right)^\alpha$ the second part becomes
\begin{align*}
I_2&\geq \int_0^t\left(\inf_{z,\,w\in B(0,\,1)}\E|u_s(z)u_s(w)|\right)^{1+\gamma}\int_{B(0,\,1)\times B(0,\,1)}G_{t-s}(x-z)G_{t-s}(y-w)f(z,w)\,\d z\d w\d s\\
&\geq c_1K_f\int_0^t\left(\inf_{z,\,w\in B(0,\,1)}\E|u_s(z)u_s(w)|\right)^{1+\gamma} \d s.
\end{align*}
Letting
\begin{equation*}
H(s):=\inf_{x,\,y\in B(0,\,1)}\E|u_s(x)u_s(y)|,
\end{equation*}
and combining the above estimates, we have
\begin{align*}
H(t)\geq c^2\kappa^2+c_1K_f\int_0^tH(s)^{1+\gamma}\,\d s, \quad\text{for}\quad t\leq\left(\frac{1}{2}\right)^{\alpha}.
\end{align*}
By taking $\kappa$ big enough, we can make sure that $t_0$ is as small as we wish by Proposition \ref{volterra}. This finishes the proof of the proposition.
\end{proof}

\begin{proof}[\bf Proof of Theorem \ref{theo2}]
We can now easily prove the theorem. From the mild formulation and Proposition \ref{prop-colored}, we have
\begin{align*}
\E|&u_t(x)|^2\\
&\geq c^2\kappa^2+\int_{0}^t\int_{B(0,\,1)\times B(0,\,1)}G_{t-s}(x-y)G_{t-s}(x-w)f(y,w)\left(\E|u_s(y)u_s(w)|\right)^{1+\gamma}\d y\d w\d s\\
&\geq c^2\kappa^2+\int_{t_0}^t\int_{B(0,\,1)\times B(0,\,1)}G_{t-s}(x-y)G_{t-s}(x-w)f(y,w)\left(\E|u_s(y)u_s(w)|\right)^{1+\gamma}\d y\d w\d s\\
&=\infty,
\end{align*}
when $\kappa$ is large.
\end{proof}

\section{Proof of Theorem \ref{Theo3}}\label{sect-4}
Before presenting the proof of our theorem, we need to give some important results given in the propositions bellow. In the remainder of this section,  $u_t$ will be the solution to \eqref{eq:colored} and
the correlation function is always given by the Riesz kernel, that is
\begin{align*}
f(x,y)=\frac{1}{|x-y|^\omega},\ \ \ \omega< d\wedge (\alpha\beta^{-1}).
\end{align*}

\begin{proposition}
For $x,\,y\in B(0,\,t^{\beta/\alpha}),$ there exists a constant $c$ such that
\begin{align*}
\int_{\R^d\times\R^d}G_t(x-z)G_t(y-w)f(z,w)\,\d z\d w\geq \frac{c}{t^{\omega\beta/\alpha}}.
\end{align*}
\end{proposition}
\begin{proof}
By the bounds given by \eqref{heat} we observe
\begin{align*}
\int_{\R^d\times\R^d}G_t(x-z)&G_t(y-w)f(z,w)\,\d z\d w\\
&\geq \int_{B(0,\,t^{\beta/\alpha})\times B(0,\,t^{\beta/\alpha})}G_t(x-z)G_t(y-w)f(z,w)\,\d z\d w\\
&\geq \frac{c_1}{t^{2d\beta/\alpha}}\int_{B(0,\,t^{\beta/\alpha})\times B(0,\,t^{\beta/\alpha})}f(z,w)\,\d z\d w\\
&\geq \frac{c_1}{t^{2d\beta/\alpha}}\frac{c^\ast}{t^{\omega\beta/\alpha}}\int_{B(0,\,t^{\beta/\alpha})\times B(0,\,t^{\beta/\alpha})}\,\d z\d w\\
& = \frac{c}{t^{\omega\beta/\alpha}}.
\end{align*}
The last inequality follows since $|z-w|\leq  2t^{\beta/\alpha}$ for $z,\,w\in B(0,\,t^{\beta/\alpha})$.
\end{proof}
The following proposition  now easily follows by the last result.
\begin{proposition}\label{correlation-lower-bound}
For fixed $t>0$, we have
\begin{align*}
\E|u_t(x)u_t(y)|\geq ct^{(\alpha-\omega\beta)/\alpha},\quad\text{for all}\quad x,\,y\in B(0,\,t^{\beta/\alpha}),
\end{align*}
where $c$ is some constant.
\end{proposition}
\begin{proof}
Since initial condition is non-negative and  $\E|u_t(x)u_t(y)|\geq \kappa$, then by the above proposition, we get
\begin{align*}
\E|u_t(x)u_t(y)|&\geq \kappa^{1+\lambda}\int_0^t\int_{\R^d\times\R^d}G_s(x-z)G_s(y-w)f(z,w)\,\d z\d w\d s\\
&\geq ct^{(\alpha-\omega\beta)/\alpha}.
\end{align*}
\end{proof}

To give the proof of Theorem \ref{Theo3} we will need the following proposition.

\begin{proposition}
Fix $t>0$ and let $t_0\leq t/3$. Then for $x,\,y\in B(0,\,t^{\beta/\alpha})$, we have
\begin{align*}
\int_0^t\int_{\R^d\times\R^d}G_{t+t_0-s}(x-z)G_{t+t_0-s}(y-w)\E|u_s(z)u_s(w)|^{1+\gamma}f(z,w)\,\d z\,\d w\geq ct^{2(\alpha-\omega\beta)/\alpha},
\end{align*}
for some constant $c$.
\end{proposition}
\begin{proof}
First observe that, if $s\geq(t+t_0)/2$, then  $s\geq t-s+t_0$ and also, if $s\leq3(t+t_0)/4$, then $s\leq 3(t-s+t_0).$ Using this, the fact that $\E|u_t(x)u_t(y)|\geq \kappa$ and  Proposition \ref{correlation-lower-bound}, we write
\begin{align*}
\int_0^t&\int_{\R^d\times\R^d}G_{t+t_0-s}(x-z)G_{t+t_0-s}(y-w)\E|u_s(z)u_s(w)|^{1+\gamma}f(z,w)\,\d z\,\d w\,\d s\\
&\geq \kappa^{\gamma}\int_0^ts^{(\alpha-2\omega\beta)/\alpha}\int_{B(0,\,s^{\beta/\alpha})\times B(0,\,s^{\beta/\alpha})}G_{t+t_0-s}(x-z)G_{t+t_0-s}(y-w)\,\d z\,\d w \,\d s\\
&\geq \kappa^{\gamma}\int_{(t+t_0)/2}^{3(t+t_0)/4}s^{(\alpha-2\omega)/\alpha}\\
&\times\int_{B(0,\,(t+t_0-s)^{\beta/\alpha})\times B(0,\,(t+t_0-s)^{\beta/\alpha})}G_{t+t_0-s}(x-z)G_{t+t_0-s}(y-w)\,\d z\,\d w \,\d s.
\end{align*}
Note that
\begin{align*}
|x-z|&\leq t^{\beta/\alpha}+(t-s+t_0)^{\beta/\alpha}\\
&\leq (t-s+t_0+s)^{\beta/\alpha}+(t-s+t_0)^{\beta/\alpha}\\
&\leq c_1(t-s+t_0)^{\beta/\alpha},
\end{align*}
for some constant $c_1$.  The last inequality is true since $f(t) = t^{\beta/\alpha}$ is increasing for $t>0$ and $s\leq  3(t-s+t_0)$ in our last integral above. By the bound on $G_t(x)$ in \eqref{heat}, we get the following not sharp bound which is sufficient for our needs:
\begin{align*}
\int_{(t+t_0)/2}^{3(t+t_0)/4}s^{(\alpha-2\omega\beta)/\alpha}&\int_{B(0,\,(t+t_0-s)^{\beta/\alpha})\times B(0,\,(t+t_0-s)^{\beta/\alpha})}G_{t+t_0-s}(x-z)G_{t+t_0-s}(y-w)\,\d z\,\d w \,\d s\\
&\geq\int_{(t+t_0)/2}^{3(t+t_0)/4}s^{(\alpha-2\omega\beta)/\alpha}c_1(t-s+t_0)^{-d\beta/\alpha}\\
&\times\int_{B(0,\,(t+t_0-s)^{\beta/\alpha})\times B(0,\,(t+t_0-s)^{\beta/\alpha})}\,\d z\,\d w \,\d s\\
&\geq c_2t^{2(\alpha-\omega\beta)/\alpha}.
\end{align*}
\end{proof}

\begin{proof}[\bf Proof of Theorem \ref{Theo3}]
With the above propositions at  hand, we are ready to give the proof of our theorem. By the mild formulation, the fact that initial condition is bounded below and change of variables give
\begin{align}\label{time shift}
\E |u_{T+t}(x)&u_{T+t}(y)|\nonumber\\
&\geq \kappa^2+\int_0^{T+t}\int_{\R^d\times \R^d}G_{T+t-s}(x-z)G_{T+t-s}(y-w)\E|u_s(z)u_s(w)|^{1+\lambda}f(z,w)\d z\d w\d s\nonumber\\
&= \kappa^2+\int_0^{T}\int_{\R^d\times \R^d}G_{T+t-s}(x-z)G_{T+t-s}(y-w)\E|u_s(z)u_s(w)|^{1+\gamma}f(z,w)\d z\d w\d s\nonumber\\
&+\int_0^{t}\int_{\R^d\times \R^d}G_{t-s}(x-z)G_{t-s}(y-w)\E|u_{T+s}(z)u_{T+s}(w)|^{1+\gamma}f(z,w)\d z\d w\d s\nonumber\\
&\geq \kappa^2+\int_0^{T}\int_{\R^d\times \R^d}G_{T+t-s}(x-z)G_{T+t-s}(y-w)\E|u_s(z)u_s(w)|^{1+\gamma}f(z,w)\d z\d w\d s\nonumber\\
&+\int_0^{t}\int_{B(0,1)\times B(0,1)}G_{t-s}(x-z)G_{t-s}(y-w)\E|u_{T+s}(z)u_{T+s}(w)|^{1+\gamma}f(z,w)\d z\d w\d s.
\end{align}
Take $T\gg 1$ and $t\leq T/3$, so that we can use the previous Proposition to bound the second term. To bound  the third term, we use similar arguments as in the proof of Theorem \ref{theo2}. If we let
\begin{align*}
Q(s):=\inf_{x,\,y\in B(0,\,1)}\E |u_{T+s}(x)u_{T+s}(y)|,
\end{align*}
we observe
\begin{align*}
Q(t)\geq \kappa^2+cT^{2(\alpha-\omega\beta)/\alpha}+c_1\int_0^{t}Q(s)^{1+\gamma}\,\d s.
\end{align*}
It suffices to consider, the following differential equation
$$\frac{Q'(t)}{Q(t)^{1+\gamma}} = c_1,$$
with initial condition $Q(0) = \kappa^2+cT^{2(\alpha-\omega\beta)/\alpha}:=A$. Solving this equation, we get
$$\frac{1}{Q(t)^{1+\gamma}} = \frac{1}{A^\gamma} - c_1\gamma t.$$
The blow up occurs at $t = \frac{1}{c_1\gamma A^\gamma}.$ That means, as long as $\kappa$ is strictly positive, we will have blow up of $Q$ for any fixed small time; we just need to take $T$ large enough. To finish the proof we use the mild formulation and the above result to write
\begin{align}\label{time-shift-solution}
&\E |u_{T+t}(x)|^2\nonumber\\
&\geq c^2\kappa^2+\int_0^{T+t}\int_{\R^d\times \R^d}G_{T+t-s}(x-z)G_{T+t-s}(y-w)(\E|u_s(z)u_s(w)|)^{1+\gamma}f(z,w)\d z\d w\d s\nonumber\\
&\geq c^2\kappa^2+\int_T^{T+t}\int_{\R^d\times \R^d}G_{T+t-s}(x-z)G_{T+t-s}(y-w)(\E|u_s(z)u_s(w)|)^{1+\gamma}f(z,w)\d z\d w\d s\nonumber\\
&=\infty,
\end{align}
when $\kappa$ is large.
\end{proof}

\section{Proof of Theorems \ref{energy-white}, \ref{energy-colored} and \ref{Dirichlet}}\label{sect-5}
The following proposition is crucial in the proof of Theorem \ref{energy-white}.
\begin{proposition}\label{prop5.1}
Under Assumption \ref{initial}, there exist  $t^\ast,\,K>0$ such that for all $t\geq t^\ast$, we have
\begin{equation*}
\inf_{x\in B(0,\,1)}\E|u_t(x)|^2=\infty,\quad\text{for}\quad K_{u_0}>K.
\end{equation*}
\end{proposition}
\begin{proof}
By Walsh isometry, we have
\begin{align*}
\E|u_t(x)|^2&=|(\cG u)_t(x)|^2+\int_0^t\int_{\R^d} G_{t-s}^2(x-y)\E|\sigma(u_s(y))|^2\,\d y\,\d s.
\end{align*}
We can always assume that $t^\ast$ to be large. Otherwise, there is nothing to prove. So instead of looking at time $t$, we will look at $t+t_0$ and fix $t_0>0$ later. We have
\begin{align*}
\E|u_{t+t_0}(x)|^2&=|(\cG u)_{t+t_0}(x)|^2+\int_0^{t+t_0}\int_{\R^d} G_{t+t_0-s}^2(x-y)\E|\sigma(u_s(y))|^2\,\d y\,\d s\\
&\geq |(\cG u)_{t+t_0}(x)|^2+\int_{t_0}^{t+t_0}\int_{\R^d} G_{t+t_0-s}^2(x-y)\E|\sigma(u_s(y))|^2\,\d y\,\d s
\end{align*}
By substituting $S = s-t_0$ in the second part, we obtain
\begin{align*}
\E|u_{t+t_0}(x)|^2&\geq |(\cG u)_{t+t_0}(x)|^2+\int_0^t\int_{\R^d} G_{t-S}^2(x-y)\E|\sigma(u_{S+t_0}(y))|^2\,\d y\,\d S\\
&:= I_1+I_2.
\end{align*}
We will assume that $t<1$ for most of the rest of the proof. We  find a lower bound  on $I_1$ first. Let $x\in B(0,\,1)$, then we fix $t_0$ as in Proposition \ref{determ}. This gives us
\begin{align*}
I_1&\geq cK_{u_0}^2,
\end{align*}
where the constant $c$ depends on $t_0$. We now look at the second term:
\begin{align*}
I_2&\geq \int_0^t\left(\inf_{y\in B(0,\,1)}\E|u_{S+t_0}(y)|^2\right)^{1+\gamma}\int_{B(0,\,1)}G_{t-S}^2(x-y)\,\d y\,\d S\\
&\geq c_1\int_0^t\left(\inf_{y\in B(0,\,1)}\E|u_{S+t_0}(x)|^2\right)^{1+\gamma}\frac{1}{(t-S)^{d\beta/\alpha}}\,\d S
\end{align*}
For the last inequality, we used the fact that $t<1$, the fact that  $\{y\in B(0,1): \ |x-y|<t^{\beta/\alpha}\}\subset \{y\in B(0,1): \ |x-y|<1\}$, and  the inequality \eqref{heat}.

Letting $R(S):=\inf_{x\in B(0,\,1)}\E|u_{S+t_0}(x)|^2$, we obtain
\begin{equation*}
R(t)\geq cK_{u_0}^2+c_1\int_0^t\frac{R(S)^{1+\gamma}}{(t-S)^{d\beta/\alpha}}\,\d S\quad\text{for}\quad t\leq1.
\end{equation*}
Now by Proposition \ref{volterra} we have the desired result.
\end{proof}
\begin{proof}[\bf Proof of Theorem \ref{energy-white}]
Let $t>t^\ast$ where $t^\ast$ is as given in the above proposition. The proof of the theorem now follows from Walsh isometry, Jensen's inequality and Proposition \ref{prop5.1}
\begin{align*}
\E|u_t(x)|^2&\geq |(\cG u)_t(x)|^2+\int_0^t\int_{\R^d} G_{t-s}^2(x-y)\left(\E|u_s(y))|^2\right)^{1+\gamma}\,\d y\,\d s\\
&\geq |(\cG u)_t(x)|^2+\int_{t^\ast}^t\int_{B(0,\,1)} G_{t-s}^2(x-y)\left(\E|u_s(y))|^2\right)^{1+\gamma}\,\d y\,\d s\\
&=\infty.
\end{align*}
This follows since the first term of the above display is strictly positive for any $x\in {\R^d}$.
\end{proof}

\begin{proposition}\label{prop5.2}
Suppose that Assumptions \ref{color} and \ref{initial} hold. Let $u_t$ be the solution to \eqref{eq:colored}. Then, there exists a $t^\ast>0$ such that for all $t\geq t^\ast$, we have
\begin{equation*}
\inf_{x,\,y\in B(0,\,1)}\E|u_t(x)u_t(y)|=\infty,\quad\text{whenever}\quad K_{u_0}>K,
\end{equation*}
 where $K$ is some positive constant.
\end{proposition}

\begin{proof}
We can always assume that $t^\ast$ to be large like in proof of Proposition \ref{prop5.1}. So instead of looking at time $t$, we will look at $t+t_0$ and fix $t_0>0$ later. From the mild formulation and appropriate change of variables as in Proposition \ref{prop5.1}, we obtain
\begin{align*}
\E&|u_{t+t_0}(x)u_{t+t_0}(y)|\\
&\geq (\cG u)_{t+t_0}(x)(\cG u)_{t+t_0}(y)\\
&+\int_0^{t+t_0}\int_{\R^d\times\R^d}G_{t+t_0-s}(x-z)G_{t+t_0-s}(y-w)f(z-w)\E|\sigma(u_s(z))\sigma(u_s(w))|\,\d z \d w \d s\\
&\geq (\cG u)_{t+t_0}(x)(\cG u)_{t+t_0}(y)\\
&+\int_0^{t}\int_{B(0,1)\times B(0,1)}G_{t-s}(x-z)G_{t-s}(y-w)f(z-w)\E|u_{s+t_0}(z)u_{s+t_0}(w)|^{1+\gamma}\,\d z \d w \d s.
\end{align*}
The proof essentially follows the same idea as in Proposition \ref{prop5.1}. The key idea is to take $t_0$ as in Proposition \ref{determ} and set
\begin{equation*}
G(s):=\inf_{x,\,y\in B(0,\,1)}\E|u_{s+t_0}(x)u_{s+t_0}(y)|.
\end{equation*}
By following the ideas in Proposition \ref{prop-colored}, we get
\begin{align*}
G(t)\geq cK_{u_0}^2+c_1K_f\int_0^tG(s)^{1+\gamma}\d s,
\end{align*}
valid for a suitable range of $t$. Now we have the desired result using Proposition \ref{volterra}.
\end{proof}
\begin{proof}[\bf Proof of Theorem \ref{energy-colored}]
With the above Proposition, the proof of theorem is now very similar to that of Theorem \ref{energy-white}. Again by Walsh isometry, we have
\begin{align*}
\E&|u_{t}(x)|^2\\
&\geq |\cG u)_{t}(x)|^2\\
&+\int_{t^\ast}^{t}\int_{B(0,1)\times B(0,1)}G_{t-s}(x-z)G_{t-s}(y-w)f(z-w)(\E|u_s(z)u_s(w)|)^{1+\gamma}\,\d z \d w \d s.
\end{align*}
Proposition \ref{prop5.2} completes the proof since the first term of the above inequality is strictly positive for any $x\in\R^d$.
\end{proof}

To prove  Theorem \ref{Dirichlet} we will follow a similar pattern of the proof of the previous results. We emphasize that in the case of \eqref{eq:dir}, the mild solution in the sense of Walsh \cite{walsh} is given by
\begin{equation}
u_t(x)=
(\cG_B u)_t(x)+ \int_{\R^d}\int_0^t G_{B}(t-s, x-y)\sigma(u_s(y))F(\d s\,\d y).
\end{equation}
\begin{proof}[\bf Proof of Theorem \ref{Dirichlet}]
Before giving the proof of our theorem we need the following result.
\begin{align*}
\E|&u_t(x)u_t(y)|\\
&\geq(\cG_B u)_t(x)(\cG_B u)_t(y)\\
&+\int_0^t\int_{B(0,\,R)\times B(0,\,R)}G_{B}(t-s,x-z)G_{B}(t-s, y-w)f(z-w)\left(\E|u_s(z))u_s(w)|\right)^{1+\gamma}\,\d z\,\d w\,\d s\\
&:=I_1+I_2.
\end{align*}
We look at $I_1$ first. By Proposition \ref{determ-dirich}, if $x,\,y \in B(0,\,R/2)$ and $t$ is small enough, we have $I_1\geq c_1\kappa^2$. We now turn our attention to the second term.
\begin{align*}
I_2&\geq \int_0^t\left(\inf_{x,\,y\in B(0,\,R/2)}\E|u_s(x)u_s(y)| \right)^{1+\gamma}\\
&\times \int_{B(0,\,R/2)\times B(0,\,R/2)}G_{B}(t-s,x-z)G_{B}(t-s,y-w)f(z-w)\,\d z\,\d w\,\d s.
\end{align*}
Fix $\epsilon=R/4$  and Proposition \ref{prop-kernel-killed} with $t\leq \big(\frac{R}{4}\big)^\alpha$ to obtain
\begin{align*}
\int_{B(0,\,R/2)\times B(0,\,R/2)}G_{B}&(t-s,x-z)G_{B}(t-s,y-w)f(z-w)\,\d z\,\d w\\
&\geq c_1K_f.
\end{align*}
We then have
\begin{align*}
I_2&\geq c_2K_f\int_0^t\left(\inf_{x,\,y\in B(0,\,R/2)}\E|u_s(x)u_s(y)| \right)^{1+\gamma}\,\d s.
\end{align*}
If we let
\begin{equation*}
H(s):=\inf_{x,\,y\in B(0,\,R/2)}\E|u_s(x)u_s(y)|,
\end{equation*}
then we get
\begin{align*}
H(t)\geq c_1\kappa^2+c_2K_f\int_0^t H(s)^{1+\gamma}\,\d s.
\end{align*}
By comparison principle, it is enough to consider
$$\frac{H'(t)}{H(t)^{1+\gamma}} = c_2K_f,$$
with initial condition $c_1\kappa ^2.$ Hence the blowup occurs at $t = \frac{1}{(c_1\kappa^2)^\gamma\gamma c_2K_f}$. Fix any $t_0<\left(\frac{R}{2}\right)^\alpha$ and take $\kappa_0>\frac{1}{c_1^{0.5}(\gamma c_2K_ft_0)^{1/2\gamma}}$ such that for $\kappa>\kappa_0$, $H(s)=\infty$ for all $s\geq t_0$. Using the above result we can easily prove our result. Observe that
\begin{align*}
\E|&u_t(x)|^2\\
&\geq |(\cG_B u)_t(x)|^2\\
&+\int_{t_0}^t\int_{\R^d\times\R^d}G_{B}(t-s, x-z)G_{B}(t-s, x-w)f(z-w)\left(\E|(u_s(z))(u_s(y))|\right)^{1+\gamma}\,\d z\,\d w\,\d s\\
&=\infty.
\end{align*}
This is true since all the relevant terms involved in the above inequality are positive.
\end{proof}

\section{Proofs of results for reaction-diffusion type equations; Theorems \ref{thm-Dirichlet-additional-drift} and \ref{thm-additional-drift} } \label{sect-drift}

\begin{proof}[\bf Proof of Theorem \ref{thm-Dirichlet-additional-drift}]

Let $B:=B(0,R)$. Suppose by contradiction that  there is a random field solution of equation \eqref{eq:dir-additional-source}. This means that
$$
\sup_{t\geq 0}\sup_{x\in B}\E |u_t(x)|^2<\infty.
$$
Then since $\E |u_t(x)|\leq (\E |u_t(x)|^2)^{1/2}$ we have

$$
\sup_{t\geq 0}\sup_{x\in B}\E |u_t(x)|<\infty.
$$
The eigenfunctions $\{\phi_n: n\in \mathbb{N}\}$ of fractional Laplacian $-(-\Delta)^{\alpha/2}$ in $B$  form  an orthonormal basis for $L^2(B)$.
It is well-known that the first eigenfunction $\phi_1(x)>0$ for all $x\in B$.
Now we also have
\begin{equation}\label{eq:-finite-l1-norm}
\begin{split}
\sup_{t\geq 0}Q(t)&:= \sup_{t\geq 0}\int_B \E[|u_t(x)|]\phi_1(x)\d x\leq  \sup_{t\geq 0}\sup_{x\in B}\E |u_t(x)|\int_B \phi_1(x)\d x<\infty.
\end{split}
\end{equation}

Next, we start with taking expectation of  both sides of equation \eqref{eq:dir-mild-additional-drift} to get
\begin{equation}\label{expectation:dir-mild-additional-drift}
\E[u_t(x)]=
(\cG_B u_0)_t(x)+ \int_{B(0,R)}\int_0^t \E[b(u_s(y))]G_B({t-s},x,y)\d s\,\d y.
\end{equation}
Multiply both sides  of \eqref{expectation:dir-mild-additional-drift} by the first eigenfunction $\phi_1(x)$ of $-(-\Delta)^{\alpha/2} $  on the ball $B$ and integrating over  {\color{blue}$B$} we get
\begin{equation}\label{eq:eigen-integrated}
\begin{split}
Q(t)&:=\int_B \E[|u_t(x)|]\phi_1(x)\d x\\
&\geq
\int_{B}(\cG_B u_0)_t(x)\phi_1(x)\d x+ \int_{B}\bigg[\int_{B}\int_0^t \E[b(u_s(y))]G_B({t-s}, x-y)\d s\,\d y\bigg]\phi_1(x)\d x\\
&\geq \int_{B}(\cG_B u_0)_t(x)\phi_1(x)\d x+ \int_{B}\bigg[\int_{B}\int_0^t \E[|u_s(y)|^{1+\eta}]G_B({t-s}, x-y)\d s\,\d y\bigg]\phi_1(x)\d x\\
&\geq \int_{B}(\cG_B u_0)_t(x)\phi_1(x)\d x+ \int_{B}\bigg[\int_{B}\int_0^t \bigg(\E[|u_s(y)|\bigg)^{1+\eta}G_B({t-s}, x-y)\d s\,\d y\bigg]\phi_1(x)\d x,\\
&=I_1+I_2.
\end{split}
\end{equation}
where the last inequality follows from Jensen's inequality.

 We only give the proof in one of the cases below for the convenience of the reader. For other cases, see Asogwa et al. \cite{asogwa-foondun-jebessa-nane-2017}

The eigenfunctions $\{\phi_n: n\in \mathbb{N}\}$ of fractional Laplacian $-(-\Delta)^{\alpha/2}$ in $B $  form  an orthonormal basis for $L^2(B)$.
We have an eigenfunction expansion of the kernel
\begin{equation}
G_B({t}, x,y)=\sum_{n=1}E_\beta(-\mu_nt^\beta)\phi_n(x)\phi_n(y).
\end{equation}
See, for example, Chen et al.    \cite{cmn-12} and  Meerschaert et al. \cite{mnv-09}.
From this equation we can easily get
\begin{equation}\label{eq:-phi-inner-product-kernel}
\int_B  G_B({t}, x,y)\phi_1(x)\d x=  E_\beta(-\mu_1t^\beta)  \phi_1(y)
\end{equation}
It is a well-know fact that $\phi_1(x)>0$ for $x\in B$.
Now consider first $I_1$, since $u_0(x)\geq \kappa$ by assumption we obtain

\begin{equation}
\begin{split}
I_1&\geq \kappa \int_{B}\int_B  G_B({t}, x,y)\phi_1(x)\d y \d x=\kappa \int_{B}E_\beta(-\mu_1t^\beta)  \phi_1(y)\d y \\
&=\kappa E_\beta(-\mu_1t^\beta) \int_{B}  \phi_1(y)dy=C_{\kappa, \phi_1} E_\beta(-\mu_1t^\beta)
\end{split}
\end{equation}
Next we estimate  $I_2$.
By Fubini theorem and equation \eqref{eq:-phi-inner-product-kernel}
\begin{equation}
\begin{split}
I_2&=\int_0^t\int_{B}E_\beta(-\mu_1(t-s)^\beta)  \phi_1(y)\E[|u_s(y)|^{1+\eta}] \d y\,\d s
\end{split}
\end{equation}

Applying the  Jensen's inequality twice using the fact that $0<A:=\int_B \phi_1 \d x<\infty$, and by using the fact that the Mittag-leffler function is a decreasing function,
we get

\begin{equation}
\begin{split}
I_2
&\geq \int_0^t E_\beta(-\mu_1(t-s)^\beta)A\bigg[\int_{B}  \E[|u_s(y)|]  \frac{\phi_1(y)}{A}\d y\bigg] ^{1+\eta}\,\d s\\
&=A^{-\eta}\int_0^t E_\beta(-\mu_1(t-s)^\beta)\bigg[\int_{B} \E[|u_s(y)|]  \phi_1(y)\d y\bigg] ^{1+\eta}\,\d s\\
&\geq A^{-\eta}E_\beta(-\mu_1t^\beta)\int_0^t \bigg[\int_{B}  \E[|u_s(y)|]  \phi_1(y)\d y\bigg] ^{1+\eta}\,\d s
\end{split}
\end{equation}

{ We have  the uniform estimate of Mittag-Leffler function in  \cite[Theorem 4]{simon}
\begin{equation}\label{eq:ML-bounds}
\frac{1}{1+\Gamma(1-\beta)t}\leq E_\beta(-t)\leq \frac{1}{1 + \Gamma(1 + \beta)^{-1}t}\ \text{for any}\ t>0.
\end{equation}
}

Using equation \eqref{eq:ML-bounds} we get
\begin{equation}\label{eq:dir-renewal0}
\begin{split}
Q(t)&=\int_B \E[|u_t(x)|]\phi_1(x)dx\\
&\geq C_1 \frac{1}{1+\mu_1\Gamma(1-\beta)t^{\beta}} +C_2  \frac{1}{1+\mu_1\Gamma(1-\beta)t^{\beta}} \int_0^t Q(s) ^{1+\eta}\,\d s
\end{split}
\end{equation}
Hence for $t\geq 1$  we get
\begin{equation}\label{eq:dir-renewal}
\begin{split}
Q(t)&=\int_B \E[|u_t(x)|]\phi_1(x)dx\\
&\geq C_3 t^{-\beta}+C_4 t^{-\beta} \int_1^t Q(s) ^{1+\eta}\,\d s
\end{split}
\end{equation}

Set $P(t)=t^\beta Q(t)$ and multiply both sides of equation \eqref{eq:dir-renewal} by $t^\beta$ to get
\begin{equation}\label{eq:dir-renewal-final}
\begin{split}
P(t)&=t^\beta\int_B \E[|u_t(x)|]\phi_1(x)dx\\
&\geq C_3+C_4 \int_1^t \frac{(s^\beta Q(s)) ^{1+\eta}}{s^{\beta(1+\eta)}}\,\d s\\
&=C_3+C_4 \int_1^t \frac{P(s) ^{1+\eta}}{s^{\beta(1+\eta)}}\,\d s
\end{split}
\end{equation}
 Now we have three cases:
When $\beta(1+\eta)<1$, $\beta(1+\eta)>1$ and $\beta(1+\eta)=1$. We only give the proof in the first case $\beta(1+\eta)<1$. In this case it is enough to  consider the following equation
$$
\frac{P'(t)}{P^{1+\eta}(t)}=\frac{C_4}{t^{\beta(1+\eta)}}, \ \ t>1\ \ \mathrm{and }\ \ P(1)=C_3.
$$
This equation has a solution that satisfies
\begin{equation}\label{P-power-eqn}
P^{-\eta}(t)=C_3^{-\eta}-\frac{\eta C_4}{1-\beta(1+\eta)}\bigg(t^{1-\beta(1+\eta)}-1\bigg)
\end{equation}
this blows up at $t=t_0>1$ that makes the right hand side of the last equation zero;
$$
t_0^{1-\beta(1+\eta)}=1+ \frac{C_3^{-\eta}}{\eta C_4}(1-\beta(1+\eta))
$$
Since the solution $P(t)$ is a non-decreasing function, $P(t)=\infty$ for all $t\geq t_0$.

The other cases are handled similarly.

Hence by Theorem \ref{thm-Dirichlet-additional-drift-pde}  $V(t,x)= \E[|u_t(x)|]$  blows up in finite time. This is a contradiction to inequality \eqref{eq:-finite-l1-norm}.

\end{proof}

{

\begin{proof}[\bf Proof of Theorem \ref{thm-additional-drift}]
Now suppose by contradiction that  there is a random field solution of equation \eqref{eq:additional source}. This means that
$$
\sup_{t\geq 0}\sup_{x\in \R^d}\E |u_t(x)|^2 <\infty.
$$
Then since $\E |u_t(x)|\leq (\E |u_t(x)|^2)^{1/2}$ we have

\begin{equation}\label{first-moment-bound}
\sup_{t\geq 0}\sup_{x\in B(0,R)}\E |u_t(x)|<\infty.
\end{equation}
Now  we start with taking expectation of the both sides of equation \eqref{mild-additional-drift} to get
\begin{equation}\label{expectation:mild-additional-drift}
\E[u_t(x)]=
\int_{\R^d} G_t(x-y) u_0(y)dy+ \int_{\R^d}\int_0^t \E[b(u_s(y))]G({t-s},x,y)\d s\,\d y.
\end{equation}
Hence by the Jensen's inequality we get
\begin{equation}\label{expectation:mild-additional-drift-deterministic}
\begin{split}
\E[|u_t(x)|]&\geq
(G_t* u_0)(x)+ \int_{\R^d}\int_0^t [\E
|u_s(y)|]^{1+\eta}G({t-s},x,y)\d s\,\d y\\
&=I+II.
\end{split}\end{equation}
We give the proof when $\inf _{x\in \R^d}u_0(x)\geq \kappa$: in this case $I\geq \kappa.$ Since  $G({t-s},x,y)$ is a probability density function on $\R^d$ we get
\begin{equation}
II\geq \int_0^t F(s)^{1+\eta}\d s
\end{equation}
where $F(s)=\inf_{y\in \R^d}\E
|u_s(y)|$. Hence we obtain
$$
F(t)\geq \kappa+ \int_0^t F(s)^{1+\eta}\d s
$$

Now if $\kappa >0$, then blow up happens at some $t_0=\kappa^{-\eta}/
\eta$ for any $\eta>0$. Hence $V(t,x)= \E[|u_t(x)|]$ blows up in finite time. Hence we have a contradiction to equation \eqref{first-moment-bound}.

The other case in the theorem is more complicated and it follows from Theorem \ref{thm-additional-drift-pde} by making the following observation: From equation \eqref{expectation:mild-additional-drift}, the function $V(t,x)= \E[|u_t(x)|]$ is a super solution of the following deterministic equation (this follows by using the Fractional Duhamels' principle in the reverse order!)
\begin{equation}\label{intro:det-fractional-drift}
\begin{split}
\partial^\beta_tV(t,x)&=-\nu(-\Delta)^{\alpha/2} V(t,x)+I^{1-\beta}[(V(t,x))^{1+\eta}]\ \ \
 t>0 \quad\text{and}\quad x\in \R^d;\\
 V(0,x)&=u_0(x), \ \ x\in \R^d.
 \end{split}
\end{equation}

By Theorem \ref{thm-additional-drift-pde} $V(t,x)= \E[|u_t(x)|]$ blows up in finite time. Hence we have a contradiction to equation \eqref{first-moment-bound}.

\end{proof}

\section{Appendix}\label{appendix}

 In this section,  we consider the following  space-time fractional  {reaction-diffusion type }equations in $(d+1)$ dimension:
\begin{equation}\label{eq:additional source-pde}
{
\begin{split}
 \partial^\beta_tV(t, x)&=
 -\nu(-\Delta)^{\alpha/2} V(t, x)+I^{1-\beta}[ b(V(t, x))]\quad{t>0}\quad\text{and}\quad x\in \mathbb{R}^d,\\
 V(0, x)&=V_0(x) \ \ x\in \mathbb{R}^d.
 \end{split}
 }
 \end{equation}
 {
 The operator $-(-\Delta)^{\alpha/2} $ denotes the fractional Laplacian, the generator of a $\alpha$-stable L\'evy process.}  The initial condition will always be assumed to be a non-negative bounded deterministic function.  The function $b$ is a  locally Lipschitz function.

 For every given $T>0$, a mild solution to \eqref{eq:additional source-pde} on $(0, T)$  is any $V$ that satisfies the following evolution equation--this is also called the mild/integral solution of equation \eqref{eq:additional source-pde}-- which follows by the fractional Duhamel's principle \cite{umarov-12}
{
\begin{equation}\label{mild}
V(t, x)=
\int_{\R^d} G_t(x-y)V_0(y)\,\d y + \int_{\R^d}\int_0^t G_{t-s}(x-y)b(V(s, y))\d s\,\d y,
\end{equation}
}
for $0<t<T$ where
$G_t(\cdot)$ denotes the density of the time changed process $X_{E_t}$.


\begin{theorem}[Theorem 1.1 in Asogwa et al. \cite{asogwa-foondun-jebessa-nane-2017}]\label{thm-additional-drift-pde}
Suppose that $0<\eta\leq \alpha/\beta d$ and $V_0\not\equiv 0$, then there is no global solution to \eqref{eq:additional source-pde} in the sense that there exists a $t_0>0$ such that $V(t,\,x)=\infty$ for all $t>t_0$ and $x\in \R^d$.
\end{theorem}

Next result gives conditions for non-existence of global mild solutions in bounded domains.

\begin{theorem}[Theorem 1.4 in Asogwa et al. \cite{asogwa-foondun-jebessa-nane-2017}]\label{thm-Dirichlet-additional-drift-pde}
Suppose  $b $ satisfies the conditions in Assumption \ref{drift-polynomial}.
Fix $R>0$ and consider
{
\begin{equation}\label{eq:dir-additional-source-pde}
\begin{split}
 \partial^\beta_tV(t, x)&=
 -\nu(-\Delta)^{\alpha/2} V(t,x)+I^{1-\beta}[ b(V(t, x))]\quad{t>0}\quad\text{and}\quad x\in B(0,\,R),\\
 V(t, x)&= 0\ \ x\in B(0,\,R)^C\\
 V(0, x)&=V_0(x) \ \ x\in B(0,\,R).
 \end{split}
 \end{equation}
}
Here $-(-\Delta)^{\alpha/2}$ denotes the generator of  $\alpha$-stable L\'evy  process killed upon exiting the ball $B(0,\,R)$.
Suppose that $0<\eta<1/\beta-1$, then there is no global solution to \eqref{eq:dir-additional-source-pde} whenever $K_{V_0, \phi_1}:= \int_{B}V_0(x)\phi_1(x)d x>0$. For any $\eta>0$,  there is no global solution whenever $K_{V_0, \phi_1}>0$ is large enough.
\end{theorem}

The mild solution of equation \eqref{eq:dir-additional-source-pde} is given by using the fractional Duhamel's principle \cite{umarov-12} as follows

{
\begin{equation}
\begin{split}
V(t, x)&=
\int_{B(0,R)}G_B(t,x,y)V_0(y)\d y + \int_{B(0,R)}\int_0^t b(V(s, y))G_B({t-s}, x,y)\d s\,\d y,
\end{split}
\end{equation}
}
where $G_B({t}, x,y)$ is the density of $X(E_t)$ killed on the exterior of $B$.

\noindent {\bf Acknowledgements}.{ The authors thank an anonymous  referee
for  reading the paper carefully and for the many comments that improved the paper. E. Nane thanks Le Chen of University of Nevada for hosting him in November, 2018 and the very useful discussions about the blow up results for the equation with  general fractional integral of the noise and pointing out to the results in his joint paper \cite{le-chen-2019}.

}


\end{document}